\documentclass[12pt,a4paper]{article}
\usepackage{amsfonts,amsmath,amsxtra}

\newtheorem{theorem}{Theorem}
\newtheorem{corollary}{Corollary}
\newtheorem{lemma}{Lemma}

\newenvironment{definition}
{\smallskip\noindent{\bf Definition\/}:}{\smallskip\par}
\newenvironment{proposition}
{\smallskip\noindent{\bf Proposition\/}.}{\smallskip\par}
\newenvironment{example}
{\smallskip\noindent{\bf Example\/}.}{\smallskip\par}
\newenvironment{remark}
{\smallskip\noindent{\bf Remark\/}.}{\smallskip\par}
\newenvironment{remarks}
{\smallskip\noindent{\bf Remarks\/}.}{\smallskip\par}
\newenvironment{proof}
{\noindent{\bf Proof\/}.}{{ $\Box$}\smallskip\par}

\title{Combinatorial computation of the motivic Poincar\'e series}
\author{E. Gorsky}

\begin{document}

\maketitle

\begin{abstract}
We give an explicit algorithm computing the motivic generalization of the Poincar\'e series of a plane curve singularity
introduced by A. Campillo, F. Delgado and S. Gusein-Zade. It is done in terms of the embedded resolution. The result is a rational function depending of the parameter $q$, at $q=1$ it coincides with the Alexander polynomial of the corresponding link. For irreducible curves we relate this invariant to the Heegaard-Floer knot homology constructed by P. Ozsv\'ath and Z. Szab\'o. Many explicit examples are considered.
\end{abstract}

\newpage

\tableofcontents

\newpage

\section{Introduction}

In the series of articles (e.g. \cite{cdg},\cite{cdg2}) A. Campillo, F. Delgado and S. Gusein-Zade proved that the Alexander polynomial of the link of the plane curve singularity is related to the generating function arising in the purely algebraic setup.

Let $C=\cup_{i=1}^{r}C_i$ be a germ of a plane curve, $$\gamma_i:(\mathbb{C},0)\rightarrow (C_i,0)$$
are the uniformizations of its components. If $f\in \mathcal{O}=\mathcal{O}_{\mathbb{C}^2,0}$ is a germ of a function on $(\mathbb{C}^2,0)$, we define
$$v_i(f)=Ord_{0}f(\gamma_i(t)),$$
and the Poincar\'e series of the curve $C$ is defined (\cite{cdg2}) as the integral with respect to the Euler characteristic
\begin{equation}
\label{inthi}
P^{C}(t_1,\ldots,t_r)=\int_{\mathbb{P}\mathcal{O}}t_1^{v_1}\cdot\ldots\cdot {t_r}^{v_r}d\chi,
\end{equation}
where $\mathbb{P}\mathcal{O}$ denotes the projectivization of $\mathcal{O}$ as a vector space.
For example, if $C$ is irreducible, we can define the decreasing filtration 
\begin{equation}
\mathcal{O}\supset J_{1}\supset J_{2}\supset\ldots,\,\,\,\,\,\,\,\,\,\,\,\,\, J_n=\{f\in \mathcal{O}|v_1(f)\ge n\},
\end{equation}
and
\begin{equation}
\label{poi}
P^{C}(t)=\sum_{n=0}^{\infty}t^{n}\dim J_{n}/J_{n+1}.
\end{equation}

Let $\Delta^{C}(t_1,\ldots,t_n)$ denote the Alexander polynomial of the intersection of $C$ with a small sphere centered at the origin.
The theorem of Campillo, Delgado and Gusein-Zade says that if $r=1$, then
\begin{equation}
(1-t)P^{C}(t)=\Delta^{C}(t),
\end{equation}
and if $r>1$, then
$$P^{C}(t_1,\ldots,t_r)=\Delta^{C}(t_1,\ldots,t_r).$$

In \cite{cdg3} there was proposed the following natural generalization of the Poincar\'e series.
One can naturally define the motivic measure on the space of functions, and consider the following motivic integral, generalizing (\ref{inthi}):
\begin{equation}
\label{intmu}
P_{g}^{C}(t_1,\ldots,t_r)=\int_{\mathbb{P}\mathcal{O}}t_1^{v_1}\cdot\ldots\cdot {t_r}^{v_r}d\mu.
\end{equation}
If $r=1$, we can rewrite (\ref{intmu}) as the generalization of (\ref{poi}):
\begin{equation}
\label{motpoi}
P^{C}_{g}(t)=\sum_{n=0}^{\infty}t^{n}{q^{\mbox{\rm\tiny codim} J_n}-q^{\mbox{\rm\tiny codim} J_{n+1}}\over 1-q},
\end{equation}
therefore in this case one can deduce $P_{g}(t)$ from $P(t)$.
If $r$ is greater than 1, the situation becomes more complicated. 
Nevertheless, the explicit algorithm for the computation of the motivic Poincar\'e series is presented in Theorem \ref{T3}. 

\begin{definition}
The reduced motivic Poincar\'e series is the power series
\begin{equation}
\overline{P}_{g}(t_1,\ldots, t_r)=(1-qt_1)\cdot\ldots\cdot(1-qt_r)\cdot P_g(t_1,\ldots,t_r).
\end{equation}
\end{definition}

We prove that the reduced motivic Poincar\'e series satisfies the following properties.

\begin{enumerate}
 \item[{\bf 1.}]{\bf  Polynomiality.} $\overline{P}_{g}(t_1,\ldots,t_r;q)$ is a polynomial in variables $t_1,\ldots,t_r$ and $q$. We give a bound for its degree on $t_1,\ldots,t_r$.
 
 \item[{\bf 2.}]{\bf  Reduction to the Alexander polynomial.} If $n=1,$ then
 $$\overline{P}_{g}(t;q=1)=\Delta(t),$$
 where $\Delta$ denote the Alexander polynomial of the link of the corresponding plane curve singularity.
 If $n>1,$ then
 $$\overline{P}_{g}(t_1,\ldots,t_r;q=1)=\Delta(t_1,\ldots,t_r)\cdot\prod_{i=1}^{r}(1-t_i).$$
 
 \item[{\bf 3.}]{\bf  Forgetting components.} Let $C$ be a curve with $r$ components, and $C_1$ be an irreducible curve. Then
 \begin{equation}
 \label{forg}
 \overline{P}_{g}^{C\cup C_1}(t_1,\ldots,t_{r},t_{r+1}=1)=(1-q)\overline{P}_{g}^{C}(t_1,\ldots,t_r).
 \end{equation}
 If $C$ has only one component, then
 $$\overline{P}_{g}^{C}(t=1)=1.$$
 
 \medskip
 
 This property is clear from the equation (\ref{intmu}),
 but seems to be curious and, for example, does not hold for the Alexander polynomial (we cannot reconstruct the Alexander polynomial of a sublink from the Alexander polynomial of a link by setting the corresponding variable to 1).
 
 \item[{\bf 4.}]{\bf  Symmetry.} Let $\mu_{\alpha}$ be the Milnor number (\cite{book}) of $C_{\alpha}$,  let $(C_{\alpha}\circ C_{\beta})$ be the intersection index
of $C_{\alpha}$ and $C_{\beta}$, let $\mu(C)$ be the Milnor number of $C$. Let $$l_{\alpha}=\mu_{\alpha}+\sum_{\beta\neq\alpha}(C_{\alpha}\circ C_{\beta}),\,\,\,\,\,\delta(C)=(\mu(C)+r-1)/2.$$
Remark that $\sum_{\alpha=1}^{r}l_{\alpha}=2\delta(C).$

 It is known that the Alexander polynomial is symmetric in a sense that
 $$\Delta(t_1^{-1},\ldots,t_{r}^{-1})=\prod_{\alpha=1}^{r}t_{\alpha}^{1-l_{\alpha}}\cdot \Delta(t_1,\ldots,t_r),\quad r>1$$
 and
 $$\Delta(t^{-1})=t^{-\mu}\Delta(t),\quad r=1.$$
 In Theorem \ref{T4} we prove a generalization of this identities that holds for any $r$, namely,
$$\overline{P}_{g}({1\over qt_1},\ldots,{1\over qt_r})=q^{-\delta(C)}\prod_{\alpha}t_{\alpha}^{-l_{\alpha}}\cdot \overline{P}_g(t_1,\ldots,t_r).$$

 
 \item[{\bf 5.}]{\bf  Relation to the knot homology.} For irreducible curves we prove that $\overline{P}_{g}(t)$ can be related by the simple procedure to the Poincar\'e polynomial of the Heegaard-Floer knot homology constructed by P. Ozsv\'ath and Z. Szab\'o. This homology theory is a "categorification" of the Alexander polynomial, tightly related with the symplectic topology and Seiberg-Witten theory. Since the origins of our and their construction are quite far, the relation between them seems to be interesting. No conceptual proof for this fact is known, and we just use  that both answers are determined by the Alexander polynomial in the same way. 
\end{enumerate}

\bigskip

The paper is organized in the following way.  In the section 2 we recall the definition of the Poincar\'e series of a plane curve singularity. Then we recall the definition of the motivic measure on the space of functions and give, following \cite{cdg3}, two definitions of the motivic Poincar\'e series as a motivic integral and in terms of the multi-index filtration associated with the curve. We give the simple method of deduction of the motivic Poincar\'e series from the ordinary Poincar\'e series for irreducible curves. In Theorem \ref{cdg} we recall the formula from \cite{cdg3} expressing the motivic Poincar\'e series in terms of the embedded resolution of a curve. This formula is proved by Campillo, Delgado and Gusein-Zade using  thorough analysis of the geometry of the functional spaces defined by the embedded resolution of a curve.

 In the section 3 we apply Theorem \ref{cdg} to a nonsingular curve and explain step-by-step the calculation of all sums involved. It turns out to be a curious exercise, and this simplest example is a toy model for the consequent combinatorial work.

The section 4 contains several steps of the simplification of Theorem \ref{cdg}. In the result (Lemma \ref{L6}) the motivic Poincar\'e series is expressed in terms of some quantities $c_{K}(n)$. In Lemma \ref{L5} the generating function for these quantities is explicitly written in the closed form. This allows to compute the motivic Poincar\'e series.

Applying Lemma \ref{L6} directly, we get a lot of similar summands which cancel after all substitutions, but this cancellation is not clear from lemmas \ref{L5} and \ref{L6}. For example, it is not even clear, that the answer is a polynomial.

Therefore in the rest of section 4 we discuss the analogues of the identity
$$\sum_{n=0}^{\infty}t^{n}q^{n^2+3n\over 2}(q^{-n}-tq)=1$$
arising in the nonsingular case. The result of this investigation is Theorem \ref{T3}, where we formulate an explicit algorithm of calculation of the motivic Poincar\'e series. This algorithm does not involve infinite sums, and can be implemented as a short {\tt Mathematica} program. 

The algorithm is presented in the same manner as in Lemma \ref{L6}: the motivic Poincar\'e series is expressed in terms of some quantities $d_{P}(n)$, which fit into the explicitly defined generating function $H_{P}(u)$. This function is generally more 
complicated than the one from Lemma \ref{L5}, but in some examples (Lemma \ref{L9}) it has more or less compact form.

Section 5 contains a bunch of explicit answers for the curves with resolutions containing up to 3 divisors. 

In the section 6 we prove the symmetry property for the motivic Poincar\'e series (Theorem \ref{T4}). It generalizes the known symmetry property for the Alexander polynomial of a link. From the viewpoint of the algebraic geometry, it is related to the Gorenstein property of the coordinate ring of a curve (\cite{cdk}), thus it seems to be related to the Kapranov's functional equation (\cite{kapr},\cite{hein}) for the motivic zeta function of a curve. 

We prove the symmetry property by proving the analogous statements for all steps of our algorithm: the function 
$H_{P}(u)$ is symmetric, what implies some relations for its coefficients $d_{P}(n)$ and, therefore, for the Poincar\'e series. 
\medskip

The main result of the section 7 is  Theorem \ref{T6} describing the surprising relation between the motivic Poincar\'e series of an irreducible plane curve singularity and another deformation of the Alexander polynomial, namely, the Poincar\'e polynomial for the Heegaard-Floer knot homology (\cite{os2},\cite{os3}). The proof is based on the fact that in both cases the Poincar\'e polynomial (and series) is defined by the Alexander polynomial. We also give some corollaries from this fact which look more geometric.
A filtered complex of $\mathbb{Z}[U]$-modules analogous to the Ozsv\'ath-Szab\'o complex $CFL^{-}(K)$ is constructed. 
This gives an algebraic model for the minus- and hat-versions of the Heegaard-Floer complexes for algebraic knots.  

We also compare the motivic Poincar\'e series with the Heegaard-Floer homologies of two-component links, corresponding to the singularities of type $A_{2n-1}$.
\bigskip

The motivic Poincar\'e series has been independently studied by J. Moyano-Fernandez and W. Zuniga-Galindo in \cite{mozu}.
Their approach is based on the study of the multi-dimensional semigroup of the singularity instead of its resolution.
In particular, they gave alternative proofs of the Theorems \ref{T3} and \ref{T4} of this article.

\section*{Acknowledgements}

This work is partially supported by the grants RFBR-007-00593, RFBR-08-01-00110-a, NSh-709.2008.1 and the Moebius Contest fellowship for young scientists.

The author is grateful to M. Bershtein, A. Gorsky, S. Gukov, S. Gusein-Zade, G. Gusev, A. Kustarev, J. Moyano-Fernandez and W. Zuniga-Galindo  for useful discussions and remarks. Special thanks to A. Beliakova for her impressive lecture on the Heegaard-Floer  homology at the University of Zurich and to J. Rasmussen for his interest to this work.


\section{Poincar\'e series and its generalization}

\subsection{Poincar\'e series}

Let $C=\cup_{i=1}^{r}C_i$ be a reduced plane curve singularity at the origin in $\mathbb{C}^2$, and $C_i$ are its irreducible components. Let $\gamma_i:(\mathbb{C},0)\rightarrow (C_i,0)$ be the uniformizations of these components.

We define $r$ integer-valued functions on the space $\mathcal{O}=\mathcal{O}_{\mathbb{C}^2,0}$ by the formula
$$v_i(f)=\mbox{\rm Ord}_0(f(\gamma_i(t)))$$
and $\mathbb{Z}^r$-indexed filtration 
$$J_{\underline{v}}=\{f\in \mathcal{O}|v_{i}(f)\ge v_i\}.$$
Note that $J_{\underline{v}}$ are also defined for negative values of $\underline{v}$.
This filtration is decreasing in a sense that if $\underline{v}_1\prec \underline{v}_2$, then $J_{\underline{v}_1}\supset J_{\underline{v_2}}$. 
Consider the Laurent series 
$$L_{C}(t_1,\ldots,t_r)=\sum_{\underline{v}}t_1^{v_1}\ldots t_r^{v_r}\cdot\dim J_{\underline{v}}/J_{\underline{v}+\underline{1}}.$$
 
\begin{definition}(\cite{cdk}, \cite{cdg})
The Poincar\'e series of the curve $C$ is defined by the formula
$$P_{C}(t_1,\ldots,t_r)={L_{C}(t_1,\ldots,t_r)\cdot\prod_{i=1}^{r}(t_i-1)\over t_1\cdot\ldots\cdot t_r-1}.$$
\end{definition}

For example, if $r=1$, we have
$$P_{C}(t)=\sum_{v=0}^{\infty}t^{v}\cdot \dim J_{v}/J_{v+1}.$$

One can prove, that $P_{C}$ is always a power series. More geometric meaning of this definition 
is given by the following interpretation of the Poincar\'e series as an integral with respect to the Euler characteristic.

\begin{proposition}(\cite{cdg2})
Let $\mathbb{P}\mathcal{O}$ denote the projectivization of the functional space $\mathcal{O}$ as a vector space.
Then the following equation holds:
\begin{equation}
\label{intchi}
P_{C}(t_1,\ldots,t_r)=\int_{\mathbb{P}\mathcal{O}}t_1^{v_1}\cdot\ldots\cdot t_r^{v_r}d\chi.
\end{equation}
\end{proposition}

On the other hand, consider the link of $C$ -- the intersection of $C$ with a small three-dimensional sphere centred at the origin. We denote its multi-variable Alexander polynomial by $\Delta_C(t_1,\ldots,t_r)$. 
Campillo, Delgado and Gusein-Zade proved the following
\begin{theorem} (\cite{cdg2})
If $r=1$ then
\begin{equation}
P_{C}(t)(1-t)=\Delta_{C}(t),
\end{equation}
and if $r>1$ then
\begin{equation}
P_{C}(t_1,\ldots,t_r)=\Delta_{C}(t_1,\ldots,t_r).
\end{equation}
\end{theorem} 

\subsection{Motivic measure}


Let $\mathcal{O}=\mathcal{O}_{\mathbb{C}^2,0}$ be the space
of formal germs of analytic functions at the origin on the plane. 
It is the set of formal power series $f(x,y)$
(without degree 0 term). 
Let $\mathcal{O}_n$ be the space of $n$-jets of such arcs, let $\pi_{n}:\mathcal{O}\rightarrow \mathcal{O}_{n}$ be the natural projection.

Let $K_0(Var_{\mathbb{C}})$ be the Grothendieck ring of complex quasiprojective
 varieties. It is generated by the isomorphism classes
of complex quasiprojective  varieties modulo the relations
 $[X]=[Y]+[X\setminus Y],$ where $Y$ is a Zariski locally closed subset of  $X$. 
Multiplication is given by the formula $[X]\cdot [Y]=[X\times Y].$
Let $\mathbb{L}=[\mathbb{C}]\in K_0(Var_{\mathbb{C}})$ be the class of the  affine line in this ring. 

The Euler characteristic provides
a ring homomorphism 
$$\chi: K_0(Var_{\mathbb{C}})\rightarrow \mathbb{Z}.$$




Consider the ring $K_0(Var_{\mathbb{C}})[\mathbb{L}^{-1}]$
with the following filtration: $F_k$ is generated by the elements
of the type $[X]\cdot [\mathbb{L}^{-n}]$ with $n-\dim X\ge k$. Let $\mathcal{M}$
be the completion of the ring $K_0(Var_{\mathbb{C}})[\mathbb{L}^{-1}]$ corresponding
to this filtration.

On an algebra of subsets of $\mathcal{O}$ Campillo, Delgado and Gusein-Zade (\cite{cdg3}),
following the ideas of Kontsevich,  Denef and Loeser (\cite{dl}) 
constructed a measure $\mu$
with values in the ring  $\mathcal{M}$. 

\begin{definition}(\cite{cdg3})
A subset $A\subset \mathcal{O}$ is said to be cylindric if there exist  $n$ and a constructible set $A_n\subset \mathcal{O}_{n}$ such that
$A=\pi_{n}^{-1}(A_n)$. For the cylindric set $A$ define its {\it motivic measure} by the formula 
$$\mu(A)=[A_n]\cdot \mathbb{L}^{-{(n+1)(n+2)\over 2}}.$$
\end{definition}

Remark that $\dim \mathcal{O}_{n}={(n+1)(n+2)\over 2}$, hence the definition of the motivic measure is in fact independent on $n$.
In a full analogy with \cite{dl}, this measure can be extended to an countable-additive $\mathcal{M}$-valued measure on a suitable algebra of subsets of $\mathcal{O}$.    

\begin{definition}
A function $f:\mathcal{O}\rightarrow G$ with values in an abelian group $G$ is called simple, if its image is countable or finite, and for every $g\in G$ the set $f^{-1}(g)$ is measurable. Using this measure, one can define in the natural way the motivic integral for simple functions on $\mathcal{O}$ as
$$\int_{\mathcal{O}}f d\mu=\sum_{g\in G} g\cdot\mu(f^{-1}(g)),$$
if  the right hand side sum converges in  $G\otimes \mathcal{M}.$
\end{definition}

\begin{remark}
Note that for cylindric sets the Euler characteristic can be defined by the formula $\chi(A)=\chi(A_n)$. This gives a $\mathbb{Z}-$valued measure on the algebra of cylindric sets.
However, it cannot be extended to the algebra of measurable sets.  
This measure provides a notion of an integral with respect to the Euler characteristic for functions on $\mathcal{O}$ with cylindric level sets. It is clear that for such functions
$$\chi(\int_{\mathcal{O}}f d\mu)=\int_{\mathcal{O}}f d\chi .$$  
\end{remark}

Using the same construction, one can define the motivic measure on the projectivization $\mathbb{P}\mathcal{O}$ of the functional space.



As a direct generalisation of the equation (\ref{intchi}) Campillo, Delgado and Gusein-Zade proposed the following

\begin{definition}
Motivic Poincar\'e series is the motivic integral
\begin{equation}
\label{inttmu}
P_{g}^{C}(t_1,\ldots,t_r)=\int_{\mathbb{P}\mathcal{O}}t_1^{v_1}\cdot\ldots\cdot t_r^{v_r}d\mu
\end{equation}
\end{definition}

As above, this definition can be reformulated in terms of the multi-index filtration on the space of functions.
Let $q=\mathbb{L}^{-1}$ be a formal variable. Let $h(\underline{v})=\mbox{\rm \mbox{\rm codim}} J_{\underline{v}}$, and
$$L_{g}(t_1,\ldots,t_r,q)=\sum_{\underline{v}\in\mathbb{Z}^r}{q^{h(\underline{v})}-q^{h(\underline{v}+\underline{1})}\over 1-q}\cdot t_1^{v_1}\ldots t_r^{v_r}.$$
Then the following equation holds (\cite{cdg3}):
\begin{equation}
\label{motfil}
P_{g}^{C}(t_1,\ldots,t_r;q)={L_{g}^{C}(t_1,\ldots,t_r)\cdot\prod_{i=1}^{r}(t_i-1)\over t_1\cdot\ldots\cdot t_r-1}.
\end{equation}

An example of the calculation of the motivic Poincar\'e series for the singularities of type $A_{2n-1}$ directly from the equation (\ref{motfil}) is presented in the section 7.4 below.

\subsection{Irreducible case}

If $r=1$, the equation (\ref{motfil}) has a very clear form, since in this case $$P_g^{C}(t)=L_{g}^{C}(t).$$
Remark that
\begin{equation}
\label{codim}
\mbox{\rm codim} J_{v}=\dim\mathcal{O}/J_1+\dim J_1/J_2+\ldots+\dim J_{v-1}/J_v,
\end{equation}
so the series $P_{g}^{C}(t)$ can be reconstructed from the series $P_{C}(t)$. 

The functional $v(f)=\mbox{\rm Ord}_0 f(\gamma(t))$ is a valuation 
on the ring $\mathcal{O}$.
The set of values of $v$ is an integer semigroup $S=\{\sigma_1,\sigma_2,\sigma_3,\ldots\}$. For example, for the singularity $x^p=y^q$ (its link is the torus $(p,q)$ knot) we have $x(t)=t^q, y(t)=t^p$,
so the corresponding semigroup is generated by $p$ and $q$. 
The coefficient at $t^{v}$
in $P_{C}(t)$ vanishes, if $J_{v}=J_{v+1}$ (or, equivalently, $v$ does not belong to the semigroup $S$) , and equals to 1 otherwise. Therefore we have
$$P_{C}(t)=1+t^{\sigma_1}+t^{\sigma_2}+t^{\sigma_3}+\ldots.$$


Now the equation (\ref{codim}) implies  the following formula for the motivic Poincar\'e series:

\begin{equation}
\label{123}
P_{g}^{C}(t;q)=1+qt^{\sigma_1}+q^2t^{\sigma_2}+q^3t^{\sigma_3}+\ldots.
\end{equation}

\begin{example}
Consider the cusp $x^2=y^3$. Its semigroup is generated by 2 and 3, the Poincar\'e series is equal to
$$P(t)=1+t^2+t^3+t^4+\ldots,$$
the motivic Poincar\'e series is equal to
$$P_{g}(t)=1+qt^2+q^2t^3+q^3t^4+\ldots.$$
Note that 
$$P(t)(1-t)=1-t+t^2,$$
what equals to the Alexander polynomial of the trefoil knot. 
\end{example}

\subsection{Formula of Campillo, Delgado and Gusein-Zade}

In \cite{cdg3} Campillo, Delgado and Gusein-Zade gave a formula for the generalized Poincar\'e series in terms of the resolution.

Let $\pi:(X,D)\rightarrow(\mathbb{C}^2,0)$ be an embedded resolution where $D=\cup_{i=1}^{s}E_i$ is the exceptional divisor. Let $E_{i}^{\bullet}$ be $E_{i}$ without intersection points of $E_{i}$ with other components of $D$,
$E_{i}^{\circ}$ be $E_{i}^{\bullet}$ without intersection points of $E_i$ with the components of the strict transform of our curve. Let $A=(E_i\circ E_j)$ be the intersection matrix and $M=-A^{-1}.$

Let $I_{0}=\{(i,j):i<j,E_i\cap E_j=pt\}$, $K_{0}=\{1,\ldots, r\}$. For $\sigma\in I_0$, $\sigma=(i,j)$ let $i(\sigma)=i$, $j(\sigma)=j$. For $I\subset I_0$, $K\subset K_0$ let
$$\mathcal{N}_{I,K}:=\{\underline{\bf n}=(n_i,n_{\sigma}',n_{\sigma}'',\tilde{n}_{k}',\tilde{n}_{k}''):n_i\ge 0,i=1\ldots, s$$
$$n_{\sigma}',n_{\sigma}'',\sigma\in I; \tilde{n}_{k}'>0,\tilde{n}_{k}''>0,k\in K\}.$$
For $\underline{\bf n}\in \mathcal{N}_{I,K}, i=1,\ldots ,s,$ let
\begin{equation}
\label{hat}
\hat{n}_{i}=n_i+\sum_{\sigma\in I:i(\sigma)=i}n_{\sigma}'+\sum_{\sigma\in I:j(\sigma)=i}n_{\sigma}''+\sum_{k\in K:i(k)=i}\widetilde{n}_k'.
\end{equation}
Let
\begin{equation}
F(\underline{\bf n})={1\over 2}(\sum_{i,j=1}^{s}m_{ij}\hat{n}_{i}\hat{n}_{j}+\sum_{i=1}^{s}\hat{n}_{i}(\sum_{j=1}^{s}m_{ij}\chi(E_{j}^{\bullet})+1))+\sum_{k\in K}\tilde{n}_{k}'',
\end{equation}
$$\overline{F}(\underline{\hat{n}})={1\over 2}(\sum_{i,j=1}^{s}m_{ij}\hat{n}_{i}\hat{n}_{j}+\sum_{i=1}^{s}\hat{n}_{i}(\sum_{j=1}^{s}m_{ij}\chi(E_{j}^{\bullet})+1)),$$
and
$$\underline{w}(\underline{\bf n})=\sum_{i=1}^{s}\hat{n}_{i}\underline{m}_{i},v_{k}(\underline{\bf n}):=w_{i(k)}(\underline{\bf n})+\tilde{n}_{k}''.$$

\begin{theorem}(\cite{cdg3})
\label{cdg}

$$P_{g}(t_1,\ldots,t_r,q)=\sum_{I\subset I_0,K\subset K_0}\sum_{\underline{\bf n}\in \mathcal{N}_{I,K}}q^{F(\underline{\bf n})-\sum_{i=1}^{s}n_i-|I|-|K|}\cdot (1-q)^{|I|+|K|}\times$$
$$\times\prod_{i=1}^{s}\left(\sum_{j=0}^{\min\{n_i,1-\chi(E_{i}^{\circ})\}}(-1)^{j}{1-\chi(E_{i}^{\circ})\choose j}q^j\right)\cdot t^{\underline{v}(\underline{\bf n})}.$$
\end{theorem}

We briefly recall the sketch of the proof from \cite{cdg3}. Consider a function $f\in \mathcal{O}$ and its pullback $\pi^{*}f$ on the  space of resolution $X$. Now let $I(f)$ be the set of intersection points in $D$ such that there are components of the strict transform of $X$ passing through them,  $K(f)$ is the analogous set of intersection points of strict transform of $C$ with $D$. Now $n_i(f)$ is the intersection index of the strict transform of $f$ with the smooth part of $E_i$, $n_{\sigma}'$ and $n_{\sigma}''$ are intersection indices of the component of the strict transform of $f$ passing through $\sigma$ with $E_{i(\sigma)}$ and $E_{j(\sigma)}$ respectively, $\tilde{n}_{k}'$ and $\tilde{n}_{k}''$ are intersection indices of the component passing through the point $k$  with $E_{i(k)}$ and corresponding component of $C$ respectively.   

Given these sets and multiplicities, the value of the function $t_1^{v_1(f)}\cdot\ldots\cdot t_r^{v_r(f)}$ is equal to
$t^{\underline{v}(\underline{\bf n})}.$ Every summand in Theorem 2 is equal to this value multiplied by the motivic measure of the set of functions providing such set of data.


\section{Example: nonsingular curve}

Let us check that for the nonsingular curve the complicated expression from Theorem 2 coincides with the expected one.

We have one divisor and one component of the strict transform of the curve.
We have $I_{0}=\emptyset$, $K_{0}=\{1\}$. Also we have $\chi(E^{\circ})=1,\chi(E^{\bullet})=2,$ hence $1-\chi(E^{\circ})=0$. To sum over $K\subset K_0$, consider two cases:

1) $K=\emptyset$. In this case $F(n)={1\over 2}(n^2+3n)$, and we have a sum
$$\sum_{n=0}^{\infty}t^n q^{n^2+3n\over 2}\cdot q^{-n}$$

2) $K=\{1\}.$ In this case $F(n)={1\over 2}(\hat{n}^2+3\hat{n})+n''$, and we have a sum
$$\sum_{\hat{n}=1}^{\infty}q^{\hat{n}^2+3\hat{n}\over 2}t^{\hat{n}}\sum_{n=0}^{\hat{n}-1}q^{-n-1}(1-q)\sum_{n''=1}^{\infty}q^{n''}t^{n''}=\sum_{\hat{n}=1}^{\infty}q^{\hat{n}^2+3\hat{n}\over 2}t^{\hat{n}}(q^{-\hat{n}}-1)\cdot {qt\over 1-qt}.$$

Summing these two expressions, we get
$$1+\sum_{n=1}^{\infty}t^{n}q^{n^2+3n\over 2}(q^{-n}+{qt\over 1-qt}(q^{-n}-1))=
1+{1\over 1-qt}\sum_{n=1}^{\infty}t^{n}q^{n^2+3n\over 2}(q^{-n}-qt)=$$
$$1+{1\over 1-qt}(\sum_{n=1}^{\infty}t^{n}q^{n(n+1)\over 2}-\sum_{n=1}^{\infty}t^{n+1}q^{(n+1)(n+2)\over 2}).$$
In the last sum all coefficients at $t^{n}$ for $n\ge 2$ cancel, therefore
$$P_{g}(t;q)=1+{tq\over 1-qt}={1\over 1-qt}.$$


\section{Combinatorics}

\subsection{Preliminary simplification}

Let $$P_{k,n}(q)=\sum_{j=0}^{n}(-1)^{j}q^{j}{k\choose j}$$
($k$ can be negative, but $n$ should be non-negative and integer).

\begin{lemma}
\label{L1}
Let $S^{n}X$ denote the $n$th symmetric power of a space $X$. Then
$$[S^{n}(\mathbb{CP}^1-k\{pt\})]=q^{-n}P_{k-1,n}(q).$$
\end{lemma}
\begin{proof}
If $Y$ denote the union of $k$ points on $\mathbb{C}^1$, then we have 
$$S^{m}(\mathbb{CP^1})=\sqcup_{i=0}^{m}S^{i}(Y)\times S^{m-i}(\mathbb{CP}^{1}\setminus Y),$$
what is equivalent to the following multiplicativity property:
$$\sum_{n=0}^{\infty}t^{n}[S^{n}(\mathbb{CP}^1)]=\sum_{n=0}^{\infty}t^{n}[S^{n}(Y)]\cdot \sum_{n=0}^{\infty}t^{n}[S^{n}(\mathbb{CP}^1\setminus Y)].$$
Since
$$\sum_{n=0}^{\infty}t^{n}[S^{n}(\mathbb{CP}^1)]=\sum_{n=0}^{\infty}t^n[\mathbb{CP}^n]={1\over (1-t)(1-\mathbb{L}t)},$$
we get
$$\sum_{n=0}^{\infty}t^{n}[S^{n}(\mathbb{CP}^1-k\{pt\})]={(1-t)^{k-1}\over (1-\mathbb{L}t)}=$$
$$\sum_{a,b}(-1)^{a}{k-1\choose a}t^{a}\mathbb{L}^{b}t^{b}=\sum_{n=0}^{\infty}t^{n}\sum_{a=0}^{n}(-1)^{a}{k-1\choose a}\mathbb{L}^{n-a}=$$
$$\sum_{n=0}^{\infty}t^{n}q^{-n}P_{k-1,n}(q).$$
\end{proof}

Let us fix some notations.

\begin{definition}
Let
$$f_{i}(I,K)=\sum_{\sigma\in I:i(\sigma)=i}1+\sum_{\sigma\in I:j(\sigma)=i}1+\sum_{k\in K:i(k)=i}1,$$
$$f_{i}(I)=\sum_{\sigma\in I:i(\sigma)=i}1+\sum_{\sigma\in I:j(\sigma)=i}1.$$
Note that $\sum_{i=1}^{s}f_{i}(I,K)=2|I|+|K|,\sum_{i=1}^{s}f_{i}(I)=2|I|$.
\medskip

To any divisor $E_i$ we associate the factor
$$\phi_i(I,K,\hat{n})=P_{1-\chi(E_{i}^{\circ})-f_i(I,K),\hat{n}_i-f_i(I,K)},$$
and let
$$G(I,K,\hat{n})=q^{|I|}(1-q)^{|I|+|K|}\prod_{i}\phi_i(I,K,\hat{n}).$$
\end{definition}

Now we can start the simplification of the algorithm proposed in Theorem \ref{cdg}. The next two lemmas will allow us to reduce the summation over all quadruples $(n_i,n_{\sigma}',n_{\sigma}'',\widetilde{n}_k')$ to the summation by  a single variable $\hat{n}_i$ defined by (\ref{hat}).

\begin{lemma}
\label{L2}
Let us fix $\hat{n}_i$. Then
\begin{equation}
\label{primes}
\sum_{n_i,n_{\sigma}',n_{\sigma}'',\widetilde{n}_k'}q^{-n_i-f_i(I,K)}P_{1-\chi(E_{i}^{\circ}),n_i}(q)=q^{-\hat{n}_i}
\phi_i(I,K,\hat{n}).
\end{equation}
\end{lemma}

\begin{proof}
By Lemma \ref{L1} we have 
$$\sum_{n_i,n_{\sigma}',n_{\sigma}'',\widetilde{n}_k'}q^{-n_i-f_i(I,K)}P_{1-\chi(E_{i}^{\circ}),n_i}(q)=
\sum_{n_i,n_{\sigma}',n_{\sigma}'',\widetilde{n}_k'}q^{-f_i(I,K)}[S^{n_i}(E_{i}^{\circ})].$$
Consider a $n_i$-tuple of points on $E_{i}^{\circ}$,  intersection points $\sigma\in I$ such that $i(\sigma)=i$ with multiplicities $n_{\sigma}'-1$, intersection points $\sigma\in I$ such that $j(\sigma)=i$ with multiplicities $n_{\sigma}''-1$, intersection points $k\in K$ such that $i(k)=i$ with multiplicities $\widetilde{n}_{k}'-1$. We get the unordered $\hat{n}_i-f_{i}$-tuple of points on $E_i^{\circ}\cup f_i(I,K)$.  Thus the sum (\ref{primes}) equals to
$$q^{-f_i(I,K)}[S^{\hat{n}_i-f_i(I,K)}(E_i^{\circ}\cup f_i(I,K))]=q^{-\hat{n}_i}P_{1-\chi(E_{i}^{\circ})-f_{i}(I,K),\hat{n}_{i}-f_{i}(I,K)}(q).$$
\end{proof}

\begin{lemma}
\label{L3}
\begin{equation}
\label{simp1}
P_{g}(t_1,\ldots,t_r,q)=\sum_{I\subset I_0,K\subset K_0}\sum_{\hat{n}_{i}\ge f_{i}(I,K)}\underline{t}^{M\underline{\hat{n}}}q^{\overline{F}(\underline{\hat{n}})} 
\prod_{i=1}^{s}q^{-\hat{n}_i}\phi_i(I,K,\hat{n})\times
\end{equation}
$$q^{|I|}(1-q)^{|I|+|K|}\prod_{k\in K}{qt_k\over 1-qt_k}.$$
\end{lemma}

\begin{proof}
First, remark  that for every $k$
$$\sum_{\tilde{n}_{k}''>0}q^{\tilde{n}_{k}''}t_{k}^{\tilde{n}_{k}''}={t_{k}q\over 1-t_{k}q},$$
so from now on we can forget about summation over $\tilde{n}_{k}''$.

We have $$q^{-\sum_{i=1}^{s}n_i-|I|-|K|}=q^{|I|}\prod_{i=1}^{s}q^{-n_i-f_i(I,K)},$$
therefore we can reformulate the statement of Theorem \ref{cdg} in the form 
$$P_{g}(t_1,\ldots,t_r,q)=\sum_{I\subset I_{0},K\subset K_0}q^{|I|}(1-q)^{-|I|}\sum_{\hat{n}_{i}\ge f_{i}(I,K)}\underline{t}^{M\underline{\hat{n}}}q^{\overline{F}(\underline{\hat{n}})}\times$$
$$\prod_{i=1}^{s}\left[\sum_{n_i,n_{\sigma}',n_{\sigma}'',\widetilde{n}_k'}q^{-n_i-f_i(I,K)}P_{1-\chi(E_{i}^{\circ}),n_i}(q) \right].$$
Now the equation (\ref{simp1}) follows from the Lemma \ref{L2}.
\end{proof}

\begin{definition}
By the reduced motivic Poincar\'e series from now on we mean
$$\overline{P}_{g}(t_1,\ldots,t_r)=P_{g}(t_1,\ldots,t_r)\cdot\prod_{j=1}^{r}(1-t_jq).$$ 
\end{definition}

\begin{lemma}
\label{L4}
\begin{equation}
\sum u^{\hat{n}}G(K,I,\hat{n})=q^{|I|}(1-q)^{|I|+|K|}\prod_{i}{u_i^{f_i(K,I)}\over 1-u_i}(1-u_iq)^{1-\chi(E_i^{\circ})-f_i(I,K)}
\end{equation}
\end{lemma}

The proof of this lemma can be found in the Appendix.

\begin{definition}
Let 
$$c_{K}(n)=\sum_{I}\sum_{K_1\subset K}(-1)^{|K|-|K_1|}G(K_1,I,n),$$
$$A_{K}(u)=\sum_n u^{n} c_{K}(n).$$
\end{definition}
\medskip

The next lemma provides a closed formula for the function $A_{K}(u)$,
which can be considered as a generating function for the quantities $c_{K}(n)$.

\begin{lemma}
\label{L5}
$$A_K(u)=(-1)^{|K|}\prod_{i}(1-u_i q)^{|\overline{K}\cap E_i|-1}(1-u_i)^{|K\cap E_i|-1}
\prod_{\sigma}(1-qu_{i(\sigma)}-qu_{j(\sigma)}+qu_{i(\sigma)}u_{j(\sigma)}).$$
\end{lemma}

The proof of this lemma can be found in the Appendix. The next lemma expresses the reduced motivic Poincar\'e series in terms of the quantities $c_{K}(n)$.

\begin{lemma}
\label{L6}
\begin{equation}
\label{l6}
\overline{P}_{g}(t_1,\ldots,t_r,q)=\sum_{n} t^{Mn}q^{F(n)-\sum n_i}\sum_{K}t_{K}q^{|K|}c_{K}(n).
\end{equation}
\end{lemma}

\begin{proof}
From the equation (\ref{simp1}) we get
$$P_{g}(t_1,\ldots,t_r,q)=\sum_{I\subset I_0,K\subset K_0}\sum_{\hat{n}_{i}\ge f_{i}(I,K)}\underline{t}^{M\underline{\hat{n}}}q^{\overline{F}(\underline{\hat{n}})} \prod_{i=1}^{s}q^{-\hat{n}_i}\phi_i(I,K,\hat{n})\times $$
$$q^{|I|}(1-q)^{|I|+|K|}\prod_{k\in K}{qt_k\over 1-qt_k}=$$
$$\sum_{I\subset I_0,K\subset K_0}\sum_{\hat{n}_{i}\ge f_{i}(I,K)}\underline{t}^{M\underline{\hat{n}}}q^{\overline{F}(\underline{\hat{n}})} \prod_{i=1}^{s}q^{-\hat{n}_i}\phi_i(I,K,\hat{n})\times q^{|I|}(1-q)^{|I|+|K|}\prod_{k\in K}{qt_k\over 1-qt_k}=$$
$${1\over \prod_{i=1}^{n}(1-qt_i)}\sum_{\hat{n}}t^{Mn}q^{F(n)-\sum n_i}\sum_{K}t_{K}q^{|K|}\sum_{I\subset I_0}\sum_{K_1\subset K}(-1)^{|K|-|K_1|}G(K_1,I,\hat{n})=$$
$${1\over \prod_{i=1}^{n}(1-qt_i)}\sum_{\hat{n}}t^{Mn}q^{F(n)-\sum n_i}\sum_{K}t_{K}q^{|K|}c_{K}(\hat{n}).$$
\end{proof}

Lemma \ref{L6} together with Lemma \ref{L5} gives the explicit description of $\overline{P}_{g}(t)$: it is expressed in terms of some quantities $c_{K}(n)$, which fit together into the generating function $A_{K}(u)$. Lemma \ref{L5} provides a  closed formula for this generating function.

Nevertheless, as the model example with a nonsingular curve shows, lots of summands in the sum (\ref{l6}) have the same power in $t$, and for $n$ large enough we have a huge number of cancellations.

\subsection{Cancellations}

We say that a subset $K\subset K_0$ is {\it proper everywhere}, if for all $i$ $K\cap E_i$ is a proper subset of $K_0\cap E_i$. We denote the set of proper everywhere subsets by $\mathcal{P}$. For any $K\subset K_0$ let $E(K)$ be the set of divisors such that for $i\in E(K)$ the set $K\cap E_i$ is empty. Sometimes we will write $i\in P$, if $i\notin E(P)$.

Using these notations, every subset $K\subset K_0$ can be presented (uniquely) in the following way: we fix a proper everywhere subset $P(K)$ and a set of divisors $E\subset E(P(K))$ where all intersection points with $K_0$ belong to $K$.

For a set $E$ of divisors let $\Delta(E)$ be the number of pairs of intersecting divisors from $E$. 
Let $\mu_i(E)=1$, if $i\in E$ and $\mu_i(E)=0$ otherwise.

\begin{lemma}
\label{L7}
For a proper everywhere set $P$ let 
\begin{equation}
\widetilde{H}_{P}(u_1,\ldots,u_s)=\sum_{E\subset E(P)}(-1)^{|K_0\cap E|}\prod u_i^{-\sum a_{ij}\mu_j}\cdot q^{\Delta(E)}\prod_{i\in E}(q-u_i)^{k_i-1}
\prod_{i\notin (P\cup E)}(1-qu_i)^{k_i-1}
\end{equation}
$$\times \prod_{\sigma}(1-q^{1-\mu_{i(\sigma)}(E)}u_{i(\sigma)}-q^{1-\mu_{j(\sigma)}(E)}u_{j(\sigma)}+q^{1-\mu_{i(\sigma)}(E)-\mu_{j(\sigma)}(E)}u_{i(\sigma)}u_{j(\sigma)}).$$
Then the polynomial $\widetilde{H}_{P}$ is divisible by $\prod_{i\in E(P)}(1-u_i)$.
\end{lemma}

The proof of this lemma can be found in the Appendix.

The next lemma explains the relation of the function $\widetilde{H}_{P}(u_1,\ldots,u_s)$ (which is a modification of  the function $A_{K}(u)$) to the coefficients $c_{K}(n)$ defined above. It is the main technical instrument in the study of the cancellations.

\begin{lemma}
\label{L8}
$$\sum_{n}u^n\sum_{E\subset E(P)}q^{-\sum_{i\in E}n_i-\Delta(E)-\sum_{i\in E}a_{ii}-|E|}q^{|K_0\cap E|}\times c_{P\cup E}(n_i+\sum a_{ij}\mu_j(E))=$$
$$(-1)^{|P|}\prod_{i\in P}[(1-qu_i)^{k_i-p_i-1}(1-u_i)^{p_i-1}]\cdot{1\over \prod_{i\in E(P)}(1-u_i)}\widetilde{H}_{P}(u_1,\ldots,u_s).$$
\end{lemma}

The proof of this lemma can be found in the Appendix.

\begin{definition}
For a proper everywhere set $P$ define the quantities $d_{P}(n)$ by the equation
\begin{equation}
H_{P}(u)=\sum_{n}d_{P}(n)u^{n}d_{P}(n)={\prod_{i\in P}[(1-qu_i)^{k_i-p_i-1}(1-u_i)^{p_i-1}]\over \prod_{i\in E(P)}(1-u_i)}\widetilde{H}_{P}(u_1,\ldots,u_s).
\end{equation}
\end{definition}

Remark that by Lemma \ref{L7} the function $H_{P}(u)$ is polynomial in $u$, so we have only finite number of non-zero coefficients $d_{P}(n)$.

\bigskip
Combining the statements of Lemma \ref{L6} and Lemma \ref{L8}, we get the following result. 

\begin{theorem}
\label{T3}
Then
$$\overline{P}_{g}(t_1,\ldots,t_r)=\sum_{P\in\mathcal{P}}(-1)^{|P|}q^{|P|}t_{P}\times \sum_{n}d_{P}(n)t^{Mn}q^{F(n)-\sum n_i}.$$
\end{theorem}

\begin{proof}
From Lemma \ref{L6} we have
$$\overline{P}(t)=\sum_{n_1}t^{Mn_1}q^{F(n_1)-\sum{n}_i}\sum_{K\subset K_0}t_{K}q^{|K|}c_{K}(n_1)=$$
$$\sum_{P\in \mathcal{P}}q^{|P|}t_{P}\sum_{n_1}t^{Mn_1}q^{F(n_1)-\sum{n}_i}\sum_{E\subset E(P)}t_{E}q^{|K_0\cap E|}c_{P\cup E}(n_1).$$
Let us collect the coefficient at $t^{Mn}$. We have
$$Mn_1+\sum \mu_j(E)=Mn,\,\,\,\, n_1=n+\sum a_{ij}\mu_{j}(E).$$
and
$$(\overline{F}(n)-\sum n_i)-(\overline{F}(n_1)-\sum n_{1i})={1\over 2}[-2\sum m_{ij}n_{i}a_{js}\mu_{j}(E)$$
$$-\sum m_{ij}a_{is}\mu_s(E)a_{jl}\mu_l(E)-\sum m_{ij} \chi(E_{i}^{\bullet})a_{js}\mu_s(E)+\sum a_{ij}\mu_j(E)].$$
Remark that $$\sum_{i\neq j}a_{ij}=2-\chi(E_j^{\bullet}),$$ hence
$$(\overline{F}(n)-\sum n_i)-(\overline{F}(n_1)-\sum n_{1i})=\sum_{i\in E}n_i+\Delta(E)+\sum_{i\in E}a_{ii}+|E|.$$
Thus
$$\overline{P}(t)=\sum_{P\in\mathcal{P}}q^{|P|}t_{P}\sum_{n}t^{Mn}q^{F(n)-\sum n_i}\sum_{E\subset E(P)}q^{-\sum_{i\in E}n_i-\Delta(E)-\sum_{i\in K}a_{ii}-|E|}$$
$$\times q^{|K_0\cap E|}c_{P\cup E}(n+\sum a_{ij}\mu_j(E)).$$
Now we apply Lemma \ref{L8}.

\end{proof}

\begin{corollary}
The power series $\overline{P}_{g}(t_1,\ldots,t_r)$ is a polynomial. 
\end{corollary}

\subsection{The algorithm }

If every line $E_i$ is intersected by the one component of the strict transform, any proper everywhere set should be empty.
Therefore we get the following statement as a corollary of Theorem \ref{T3}. 

\begin{lemma}
\label{L9}
Suppose that each divisor $E_i$ is intersected by exactly one component of the strict transform of the curve.
Then the reduced motivic Poincar\'e series can be computed using the following algorithm.

1. Consider the polynomial
$$A(u_1,\ldots,u_r)=\prod_{\sigma}(1-qu_{i(\sigma)}-qu_{j(\sigma)}+qu_{i(\sigma)}u_{j(\sigma)}).$$

2. Consider the Laurent polynomial
$$\widetilde{H}(u_1,\ldots,u_t)=\sum_{K\subset K_0}(-1)^{|K|}q^{\Delta(K)}\prod u_{i}^{-\sum a_{ij}\mu_j}\cdot A(u_1q^{-\mu_1(K)},\ldots,u_rq^{-\mu_r(K)}).$$

3. This polynomial is divisible by $\prod (1-u_i)$. Let
$$H(u_1,\ldots,u_r)={\widetilde{H}(u_1,\ldots,u_r)\over \prod_{i=1}^{r}(1-u_i)}.$$

4. Expand this polynomial:
$$H(u_1,\ldots,u_r)=\sum d_{\underline{n}}u^{\underline{n}},$$
and now
$$\overline{P}_{g}(t_1,\ldots,t_r)=\sum d_{\underline{n}}t^{M\underline{n}}q^{F(\underline{n})-\sum n_{i}}.$$
\end{lemma} 





\section{Examples}

\subsection{One divisor}

We consider the singularity $$x^{k_0}-y^{k_0}=0,$$ which is geometrically a union of $k_0$ pairwise transversal lines.
Its minimal resolution has one divisor and $k_0$ components of the strict transform intersecting it.  
In particular, for $k_0=1$ we get a non-singular case considered above.
For $0<k<k_0$ let the numbers $c_k(n)$ be defined by the equation
$$A_k(u)=\sum_{n=0}^{\infty}u^{n}c_k(n)=(1-uq)^{k_0-k-1}(1-u)^{k-1},$$
and for $k=0$ let the numbers $c_0(n)$ be defined by the equation
$$A_0(u)=\sum_{n=0}^{\infty}u^{n}c_{0}(n)={(1-uq)^{k_0-1}-u(u-q)^{k_0-1}\over 1-u}.$$
The polynomials $A_k(u)$ have degree $k_0-2$ for $k>0$, $A_0(u)$ has degree $k_0-1$, so we have a finite number of non-zero $c_k(n)$.

From the Theorem \ref{T3} we conclude that
$$\overline{P}_{g}(t_1,\ldots,t_{k_0})=\sum_{K\subset_{\neq} K_0}(-1)^{|K|}q^{|K|}t_{K} \sum_{n=0}^{\infty}c_{|K|}(n)(t_1\ldots t_{k_0})^{n}q^{n(n+1)\over 2}.$$

For example, if $k_0=2$, $$A_1(u)=1, A_0(u)={1-uq-u(u-q)\over 1-u}=1+u,$$
so
$$\overline{P}_{g}(t_1,t_2)=1-qt_1-qt_2+qt_1t_2.$$

If $k_0=3$, $$A_1(u)=1-qu, A_2(u)=1-u, A_0(u)=1+(1-2q-q^2)u+u^2,$$
so
$$\overline{P}_{g}(t_1,t_2,t_3)=1-q(t_1+t_2+t_3)+q^2(t_1t_2+t_1t_3+t_2t_3)+q(1-2q-q^2)t_1t_2t_3+$$
$$q^3t_1t_2t_3(t_1+t_2+t_3)-q^3t_1t_2t_3(t_1t_2+t_1t_3+t_2t_3)+q^3t_1^2t_2^2t_3^2.$$
This answer can be rewritten as
$$\overline{P}_{g}(t_1,t_2,t_2)=(1-qt_1)(1-qt_2)(1-qt_3)-q^3t_1t_2t_3(1-t_1)(1-t_2)(1-t_3)+q(1-q)^2t_1t_2t_3.$$

\subsection{Two divisors}

Suppose that the second divisor is intersected by two components of the strict transform, and the first one by one component. This corresponds to the singularity
$$x\cdot (y-x^2)\cdot (y+x^2)=0.$$
The matrix $M$ is equal to
$$M=\left(
\begin{matrix}
1 &1\\
1 &2\\
\end{matrix}\right),
$$
$$\chi(E_1^{\bullet})=\chi(E_2^{\bullet})=1,$$
so
$$F(n_1,n_2)={1\over 2}(n_1^2+2n_1n_2+2n_2^2+2n_1+3n_2).$$
If $P=\emptyset$, we get
$$\widetilde{H}_{\emptyset}(u_1,u_2)=(1-qu_1-qu_2+qu_1u_2)(1-qu_2)-(1-u_1-qu_2+u_1u_2)(1-qu_2)u_1^2u_2^{-1}$$
$$+(1-qu_1-u_2+u_1u_2)(q-u_2)u_1^{-1}u_2-q(1-u_1-u_2+q^{-1}u_1u_2)(1-qu_2)u_1=$$
$${1\over u_1u_2}(1-u_1)(1-u_2)(-u_1^3 + u_1u_2 + u_1^2u_2 - qu_1^2u_2 - q^2u_1^2u_2 + 
      qu_1^3u_2$$
      $$ + qu_2^2 + u_1u_2^2 - qu_1u_2^2 - q^2u_1u_2^2 + 
      u_1^2u_2^2 - u_2^3),$$
if $P$ is one point on the second divisor, we get
$$\widetilde{H}_{pt}(u_1,u_2)=(1-qu_1-qu_2+qu_1u_2)-(1-u_1-qu_2+u_2)u_1^2u_2^{-1}=$$
$$-{1\over u_2}(1-u_1)(u_1^2 - u_2 - u_1u_2 + qu_1u_2 - u_1^2u_2 + qu_2^2).$$
Finally we get the following answer ($t_0$ corresponds to the first divisor):
$$\overline{P}_{g}(t_0,t_1,t_2)=1 - qt_0 - qt_1 + q^2t_0t_1 - qt_2 + q^2t_0t_2 + q^2t_1t_2 + 
    qt_0t_1t_2 - q^2t_0t_1t_2 - q^3t_0t_1t_2$$
    $$ - q^2t_0t_1^2t_2 + q^3t_0t_1^2t_2
    - q^2t_0t_1t_2^2 + 
    q^3t_0t_1t_2^2 + q^2t_0t_1^2t_2^2 - q^3t_0t_1^2t_2^2 - 
    q^4t_0t_1^2t_2^2 + q^4t_0^2t_1^2t_2^2$$
    $$ + 
    q^4t_0t_1^3t_2^2 - q^4t_0^2t_1^3t_2^2 + 
    q^4t_0 t_1^2t_2^3 - q^4t_0^2t_1^2t_2^3 - 
    q^4t_0t_1^3t_2^3 + q^4t_0^2t_1^3t_2^3.$$
  This answer can be rewritten as
  $$\overline{P}_{g}(t_0,t_1,t_2)=(1-qt_0)(1-qt_1)(1-qt_2)-q^4t_0t_1^2t_2^2(1-t_0)(1-t_1)(1-t_2)$$
  $$+(1-q)qt_0t_1t_2(1-qt_1-qt_2+qt_1t_2).$$
  If $q=1$, we get the known Alexander polynomial:        
$$\overline{P}_{g}(t_0,t_1,t_2;q=1)=(1-t_0)(1-t_1)(1-t_2)(1-t_0t_1^2t_2^2).$$
If $t_2=1$, we get the known answer for $A_1$ singularity:
$$\overline{P}_{g}(t_0,t_1,1)=(1-q)(1-qt_0-qt_1+qt_0t_1).$$
If $t_0=1$, we get the answer for $A_3$ singularity:
$$\overline{P}_{g}(1,t_1,t_2)=(1-q)(1 - qt_1 - qt_2 + qt_1t_2 + q^2t_1t_2 - q^2t_1^2 t_2 - 
      q^2t_1t_2^2 + q^2t_1^2t_2^2),$$
so
$$\overline{P}_{g}^{A_3}(t_1,t_2)=(1-qt_1)(1-qt_2)+qt_1t_2(1-qt_1-qt_2+qt_1t_2)=$$
$$(1-qt_1)(1-qt_2)+q^2t_1t_2(1-t_1)(1-t_2)+(1-q)qt_1t_2.$$

This answer agrees with the general answer for the singularities of type $A_{2n-1}$ in the section 7.5.

\subsection{Three divisors}

For simplicity we assume that each divisor is intersected by one component of the strict transform.
This corresponds to the singularity $$x\cdot y\cdot (x^2-y^3)=0.$$
Matrix $M$ is equal to
$$M=\left(
\begin{matrix}
1 &1 &2\\
1 &2 &3\\
2 &3 &6\\	
\end{matrix}\right),
$$
$$\chi(E_1^{\bullet})=\chi(E_2^{\bullet})=1, \chi(E_3^{\bullet})=0,$$
so
$$F(n_1,n_2,n_2)={1\over 2}(n_{1}^{2}+2n_{2}^{2}+6n_{3}^{2}+2n_{1}n_{2}+4n_{1}n_{3}+6n_{2}n_{3}+n_{1}+2n_{2}+4n_{3}).$$
Now 
$$A(u_1,u_2,u_3)=(1-qu_1-qu_3+qu_1u_3)(1-qu_2-qu_3+qu_2u_3),$$
so
$$E(u_1,u_2,u_3)={1\over u_1u_2u_3^2}( {u_2}^{3}u_3u_1-{u_1}^{3}{u_3}^{2}q+{u_1}^{4}u_3u_2-{u_1}^{2}{u_2}^{2}{u_3}^{2}-{u_2}^{2}{u_3}^{2}u_1+$$
$${u_1}^{4}{u_2}^{3}u_3-{u_3}^{3}{u_1}^{2}q-{u_1}^{3}u_2{u_3}^{2}+{u_1}^{3}{u_2}^{3}u_3+{u_1}^{2}{u_2}^{3}u_3-{u_3}^{3}qu_2-$$
$${u_1}^{3}{u_2}^{2}{u_3}^{2}-{u_3}^{3}u_1q-{u_2}^{2}{u_3}^{2}q-{u_1}^{2}u_2{u_3}^{2}-{u_3}^{2}u_1u_2+{u_2}^{2}{u_1}^{4}u_3-{u_1}^{3}{u_2}^{3}qu_3+$$
$${u_2}^{2}{u_3}^{2}{u_1}^{2}q-{u_1}^{4}u_3{u_2}^{2}q-{u_1}^{4}{u_3}^{2}u_2q-{u_2}^{3}{u_3}^{2}u_1q-
{u_2}^{3}u_3{u_1}^{2}q+{u_3}^{3}u_1{q}^{2}u_2+$$
$${u_2}^{2}{u_3}^{2}u_1{q}^{2}+{u_1}^{3}{u_2}^{2}u_3{q}^{2}+{u_1}^{3}{u_3}^{2}u_2{q}^{2}-{u_1}^{4}{u_2}^{3}+{u_1}^{2}{u_3}^{3}+{u_3}^{3}u_1+{u_3}^{2}{u_1}^{2}u_2q+$$
$${u_1}^{3}{u_3}^{3}+{u_3}^{3}{u_2}^{2}+{u_3}^{3}u_2+{u_3}^{3}-{u_3}^{4}),$$
and
$$\overline{P}_{g}(t_1,t_2,t_3)=1-t_{{3}}q+{t_{{1}}}^{2}{t_{{2}}}^{3}{t_{{3}}}^{7}{q}^{7}+{t_{{1}}}^{2}{t_{{2}}}^{2}{t_{{3}}}^{5}{q}^{5}\\
\mbox{}+t_{{1}}t_{{2}}{t_{{3}}}^{3}{q}^{3}+t_{{1}}{t_{{2}}}^{2}{t_{{3}}}^{4}{q}^{4}-{t_{{1}}}^{2}{t_{{2}}}^{4}{t_{{3}}}^{7}{q}^{7}+$$
$$t_{{2}}t_{{3}}{q}^{2}-t_{{1}}t_{{2}}{t_{{3}}}^{3}{q}^{2}+t_{{1}}t_{{2}}{q}^{2}\\
\mbox{}-t_{{1}}{t_{{2}}}^{2}{t_{{3}}}^{4}{q}^{3}-{t_{{1}}}^{2}{t_{{2}}}^{2}{t_{{3}}}^{5}{q}^{4}\\
\mbox{}-t_{{1}}{t_{{2}}}^{2}{t_{{3}}}^{2}{q}^{2}-{t_{{1}}}^{2}{t_{{2}}}^{3}{t_{{3}}}^{5}{q}^{5}-$$
$${t_{{1}}}^{3}{t_{{2}}}^{3}{t_{{3}}}^{7}{q}^{7}-{t_{{1}}}^{3}{t_{{2}}}^{4}{t_{{3}}}^{6}{q}^{7}\\
\mbox{}+{t_{{1}}}^{2}{t_{{2}}}^{3}{t_{{3}}}^{5}{q}^{4}+{t_{{1}}}^{2}{t_{{2}}}^{2}{t_{{3}}}^{4}{q}^{3}\\
\mbox{}+{t_{{1}}}^{2}{t_{{2}}}^{2}{t_{{3}}}^{3}{q}^{4}-{t_{{1}}}^{2}{t_{{2}}}^{2}{t_{{3}}}^{3}{q}^{3}+$$
$${t_{{1}}}^{2}{t_{{2}}}^{3}{t_{{3}}}^{4}{q}^{5}+{t_{{1}}}^{2}{t_{{2}}}^{4}{t_{{3}}}^{6}{q}^{7}\\
\mbox{}+t_{{1}}{t_{{2}}}^{2}{t_{{3}}}^{2}{q}^{3}-{t_{{1}}}^{2}{t_{{2}}}^{3}{t_{{3}}}^{6}{q}^{7}\\
\mbox{}-{t_{{1}}}^{2}{t_{{2}}}^{2}{t_{{3}}}^{4}{q}^{5}-t_{{1}}{t_{{2}}}^{2}{t_{{3}}}^{3}{q}^{4}-$$
$$t_{{1}}t_{{2}}t_{{3}}{q}^{3}+t_{{1}}{t_{{2}}}^{2}{t_{{3}}}^{3}{q}^{2}-t_{{2}}q\\
\mbox{}+t_{{1}}t_{{3}}{q}^{2}-t_{{1}}t_{{2}}{t_{{3}}}^{2}{q}^{2}+{t_{{1}}}^{3}{t_{{2}}}^{4}{t_{{3}}}^{7}{q}^{7}+$$
$$t_{{1}}t_{{2}}{t_{{3}}}^{2}q-t_{{1}}q-{t_{{1}}}^{2}{t_{{2}}}^{3}{t_{{3}}}^{4}{q}^{4}\\
\mbox{}+{t_{{1}}}^{3}{t_{{2}}}^{3}{t_{{3}}}^{6}{q}^{7}.$$

It can be rewritten as

$$\overline{P}_{g}(t_1,t_2,t_3)=(1-t_{{1}}q)(1-t_{{2}}q)(1-t_{{3}}q)-{t_1}^2{t_2}^3{t_3}^6q^7(1-t_1)(1-t_2)(1-t_3)-$$
$$t_1t_2t_3^2q(q-1)(1-t_2q)(1-t_3q)-t_1^2t_2^2t_3^4q^4(q-1)(1-t_2)(1-t_3)-$$
$$t_1t_2^2t_3^3q^2(q-1)(1-t_1q)+t_1t_2^2t_3^4q^3(q-1)(1-t_1).$$

In this presentation the symmetry of $\overline{P}_g$ is clear, since every line in the right hand side is invariant under the change $t_i\leftrightarrow q^{-1}t_i^{-1}$.

If we set $q=1$, we get
$$\overline{P}_{g}(t_1,t_2,t_3,q=1)=(1-t_1^{2}t_2^{3}t_3^{6})(1-t_1)(1-t_2)(1-t_3).$$

If we consider only singularity of type $A_2$, we set
$t_1=t_2=1,t_3=t$, and
$$\overline{P}_{g}(1,1,t)=(1-q)^2(1-tq+t^2q),$$
so
$$P_{g}(1,1,t)={1-tq+t^2q\over 1-tq}=1+\sum_{k=2}^{\infty}t^{k}q^{k-1}.$$
This answer coincides with the one obtained in the section 2.3.


\section{Symmetry}

In this section we prove the symmetry property for the reduced motivic Poincar\'e series (Theorem \ref{T4}).
The strategy of the proof passes along the lines of the computation described in Lemma \ref{L6}: namely, we 
prove the symmetry property for the generating function $A_{K}(u)$ in Lemma \ref{L10}, deduce from it a certain
relations on its coefficients $c_{K}(n)$ in Lemma \ref{L11}. Since we can express the motivic Poincar\'e series in terms of $c_{K}(n)$, we can finish the proof by fitting this relations to the statement of Theorem \ref{T4}.

\begin{lemma}
\label{L10}
$$A_K({1\over qu_1},\ldots,{1\over qu_s})=q^{1-|K|}\prod_{i=1}^{s} u_i^{\chi(E_i^{\circ})}\cdot A_{\overline{K}}(u_1,\ldots,u_s).$$
\end{lemma}

\begin{proof}
$$A_{K}({1\over qu})=(-1)^{|K|}\prod_{i}(1-{1\over u_i})^{|\overline{K}\cap E_i|-1}(1-{1\over u_iq})^{|K\cap E_i|-1}\prod_{\sigma}(1-{1\over u_{i(\sigma)}}-{1\over u_{j(\sigma)}}+{1\over qu_{i(\sigma)}u_{j(\sigma)}})=$$
$$A_{\overline{K}}(u)\prod_{i}u_{i}^{1-|\overline{K}\cap E_i|}u_i^{1-|K\cap E_i|}q^{1-|K\cap E_i|}\prod_{\sigma}(qu_{i(\sigma)}u_{j(\sigma)})^{-1}=$$
$$A_{\overline{K}}(u)q^{s-|K|-|I_0|}\prod u_i^{2-|K_0\cap E_i|+\chi(E_i^{\bullet})-2}.$$
It rests to note that $|I_0|=s-1$ and $\chi(E_{i}^{\circ})=\chi(E_i^{\bullet})-|K_0\cap E_i|.$
\end{proof}

\begin{lemma}
\label{L11}
$$c_{K}(n_1,\ldots,n_s)=q^{1-|K|+n}c_{\overline{K}}(-\chi(E_1^{\circ})-n_1,\ldots,-\chi(E_s^{\circ})-n_s),$$
where $n=\sum_{i=1}^{s} n_i$.
\end{lemma}

\begin{proof}
$$A_{K}({1\over qu_1},\ldots,{1\over qu_s})=\sum_{\underline{n}} c_{K}(n_1,\ldots,n_s)\underline{u}^{-\underline{n}}q^{-n}=q^{1-|K|}\prod u_i^{\chi(E_i^{\circ})}\sum_{\underline{z}} c_{\overline{K}}(z_1,\ldots,z_s)\underline{u}^{\underline{z}}.$$
We have $$z_i+\chi(E_i^{\circ})=-n_i,\,\,\, z_i=-\chi(E_i^{\circ})-n_i.$$
\end{proof}

\begin{theorem}
\label{T4}
Let $\mu_{\alpha}$ be the Milnor number of $C_{\alpha}$, and $(C_{\alpha}\circ C_{\beta})$ is the intersection index
of $C_{\alpha}\circ C_{\beta}$, $\mu(C)$ is the Milnor number of $C$. Let $l_{\alpha}=\mu_{\alpha}+\sum_{\beta\neq\alpha}(C_{\alpha}\circ C_{\beta})$ and $\delta(C)=(\mu(C)+r-1)/2.$
Then
$$\overline{P}_{g}({1\over qt_1},\ldots,{1\over qt_r})=q^{-\delta(C)}\prod_{\alpha}t_{\alpha}^{-l_{\alpha}}\cdot \overline{P}_g(t_1,\ldots,t_r).$$
\end{theorem}

The theorem follows from Lemma \ref{L11} describing the symmetry of the coefficients $c_{K}(n)$ and Lemma \ref{L6} describing $\overline{P}_g(t_1,\ldots,t_r)$ in terms of $c_{K}(n)$. The detailed proof is rather technical and can be found in the Appendix.

\begin{corollary}
The degree of the polynomial $\overline{P}_{g}(t_1,\ldots,t_r)$ with respect to the variable $t_i$ is equal to $l_i$.
The greatest monomial in it equals to $q^{\delta(C)}\prod_{i=1}^{r}t_i^{l_i}$.
\end{corollary}

\medskip

Alternative proof of the symmetry property for the motivic Poincar\'e series can be found in \cite{mozu},
where it is deduced from the theorem of Campillo, Delgado and Kiyek on the symmetry of the multi-variable 
Poincar\'e series of a plane curve singularity.

\section{Relation to the Heegaard-Floer knot homology}

\subsection{Heegaard-Floer homology}

In the series of articles (e.g. \cite{os2},\cite{os3},\cite{os4},\cite{os}, see also \cite{ras}) P. Ozsv\'ath and Z. Szab\'o constructed new powerful knot invariants, Heegaard-Floer knot (and link) homology. To each link $L=\cup_{i=1}^{r}K_i$ they assign 
the collection of homology groups $\widehat{HFL}_{d}(L,\underline{h})$, where $d$ is an integer and $\underline{h}$ belongs to some $r$-dimensional lattice. Their original description was based on the constructions from the symplectic topology, later (\cite{mos},\cite{most}) there were elaborated combinatorial models for them.
All of these homologies are invariants of the link $L$, and they have the following properties (\cite{os3}, \cite{most}).

First, they give a "categorification" of the Alexander polynomial of $L$: if $r=1$, then
$$\sum_{h}\chi(\widehat{HFL}_{*}(L,h))t^{h}=\Delta^{s}(t),$$
where $\Delta^{s}(t)=t^{-\deg \Delta/2}\Delta(t)$ is a symmetrized Alexander polynomial of $L$.
If $r>1$, then
$$\sum_{\underline{h}}\chi(\widehat{HFL}_{*}(L,h))\underline{t}^{h}=\prod_{i=1}^{r}(t_i^{1/2}-t_i^{-1/2})\cdot \Delta^{s}(t_1,\ldots,t_r).$$

Second, they have the symmetry extending the symmetry of the Alexander polynomial:
$$\widehat{HFL}_d(L, h)\cong \widehat{HFL}_{d-2H}(L,-h),$$
where $H=\sum_{i=1}^{r}h_i$.  

These properties are  similar to the ones of the polynomials $\overline{P}_{g}(t)$, and one could be interested in comparison of these objects. It turns out, that for knots (of course, $\overline{P}_{g}(t)$ is defined only for the algebraic ones) this comparison can be done.  

In \cite{os} for the relatively large class of knots, containing all algebraic knots, the following statement was proved. 

\begin{theorem}(\cite{os})
\label{T5}
Let the symmetrized Alexander polynomial have the form
$$\Delta^{s}(t)=(-1)^{k}+\sum_{i=1}^{k}(-1)^{k-i}(t^{n_i}+t^{-n_i})$$
for some integers $0<n_1<n_2<\ldots<n_k$. Let $n_{-j}=-n_{j},n_0=0$.
For $-k\le i\le k$ let us introduce the numbers $\delta_{i}$ by the formula
$$\delta_{i}=\begin{cases}
0, &\mbox{if}\,\,\,\, $i=k$\cr
\delta_{i+1}-2(n_{i+1}-n_i)+1, &\mbox{if}\,\,\, $k-i$\,\,\, \mbox{is odd}\cr
\delta_{i+1}-1, &\mbox{if}\,\,\, $k-i$>0\,\,\, \mbox{is even}.\cr 
\end{cases}$$

Then $\widehat{HFL}(K,j)=0$, if $j$ does not coincide with any $n_i$, and $\widehat{HFL}(K,n_i)=\mathbb{Z}$ belongs to the homological grading $\delta_i$.
\end{theorem}

In what follows we will need more detailed algebraic structure of the Heegaard-Floer homology which can be described in the following way (\cite{os3}).

Consider the ring $$R=\mathbb{Z}[U_1,\ldots,U_r].$$
For every $r$-component link $L$ there exists a $\mathbb{Z}^r$-filtered chain complex
$CFL^{-}(S^3,L)$ of $R$-modules, whose filtered homotopy type is an invariant of the link $L$.
Filtrations naturally correspond to the components of the link $L$.
The operators $U_i$ lowers the homological grading by 2 and the filtration level by $\underline{1}$.
The homologies of the associated graded object are denoted as $HFL^{-}(S^{3},L)$. 
If one sets $U_1=U_2=\ldots=U_r=0$, he gets a new $\mathbb{Z}^r$-filtered chain complex of $\mathbb{Z}$-modules, which will be denoted as $\widehat{CFL}(L)$. The homology of the associated graded object are denoted as $\widehat{HFL}(L)$, and they are the homology discussed above.

The filtration on the second complex is compatible with the forgetting of components (proposition 7.1 in \cite{os3}).
Namely, let $M$ be the two-dimensional graded vector space with one generator in grading 0 and one in grading $-1$.

\begin{proposition}
Let $L$ be an oriented, $r$-component link in $S^3$ and distinguish the first component $K_1$. Consider the complex
$\widehat{CFL}(L)$ viewed as a $\mathbb{Z}^{n-1}$-filtered chain complex where the filtration corresponding to the first component is omitted.
The filtered homotopy type of this complex is identified with $\widehat{CFL}(L-K_1)\otimes M$.
\end{proposition}  

If we forget all components of $L$, we get either the complex $$\hat{CF}(S^{3})\otimes M^{r-1},$$ where $\hat{CF}(S^3)$ has one-dimensional homology in grading 0 or $${CF}^{-}(S^3)=\mathbb{Z}[U],$$ where all $U_i$ acts by the multiplication by $U$.
  
This proposition is a direct analogue to the equation (\ref{forg}).

\bigskip

The three-manifolds with simplest Heegaard-Floer homology are the rational homology spheres  $Y$,
for which the rank of the Heegaard-Floer homology is equal to the order of the first (singular) homology,
i.e. $$\mbox{\rm rk}\,\,\,\, \widehat{HF(Y)} = |H_1(Y ;Z)|.$$ These manifolds are called $L$-spaces, for example, lens spaces are L-spaces. 
In the case that some positive surgery on $K$ gives an $L$-space,
we call $K$ an $L$-space knot. It was proved by M. Hedden in \cite{hedden} that all algebraic knots (i.e. links of 
irreducible plane curve singularities) belong to the class of $L$-space knots.

It was proved in \cite{os}, that for the $L$-space knot $K$ and any filtration level $n$
\begin{equation}
\label{rk1}
\mbox{\rm rk}\,\,\,\, H^{*}(CFL^{-}(K,n)/U_1(CFL^{-}(K,n)))=1.
\end{equation}
This is a key geometric ingredient in the proof of  Theorem \ref{T5}.

\subsection{Matching the answers}

Consider the Poincar\'e polynomial for the Heegaard-Floer homologies:
$$HFL(t,u)=\sum u^d t^s \dim \widehat{HFL}_{d,s}(K).$$
It categorifies the Alexander polynomial in the sense that
$$HFL(t,-1)=t^{-\deg \Delta/2}\Delta(t).$$

Remark that the coefficients in $\overline{P}_{g}(t,q)$ are always equal to 0 or to $\pm 1$. It can be proved  from the equation (\ref{123}).

\begin{theorem} 
\label{T6}
Take $\overline{P}_{g}(t,q)$ and let us make a following change  in it: $t^{\alpha}q^{\beta}$ is transformed to $t^{\alpha}u^{-2\beta}$, and $-t^{\alpha}q^{\beta}$ is transformed to $t^{\alpha}u^{1-2\beta}.$
We get a polynomial $\widetilde{\Delta}_{g}(t,u)$. Then
\begin{equation}
\label{ptohfk}
\widetilde{\Delta}_{g}(t^{-1},u)=t^{-\deg \Delta/2}HFL(t,u).
\end{equation}
\end{theorem}

\begin{example}
For $(3,5)$ torus knot we have
$$P_{g}(t,q)=1+qt^3+q^2t^5+q^3t^6+{q^4t^8\over 1-qt},$$
$$\overline{P}_{g}(t,q)=1-qt+qt^3-q^2t^4+q^2t^5-q^4t^7+q^4t^8,$$
$$\widetilde{\Delta_{g}}(t,q)=1+u^{-1}t+u^{-2}t^3+u^{-3}t^4+u^{-4}t^5+u^{-7}t^{7}+u^{-8}t^{8},$$
and
$$HFL(t,u)=t^{4}+u^{-1}t^{3}+u^{-2}t+u^{-3}t^{0}+u^{-4}t^{-1}+u^{-7}t^{-3}+u^{-8}t^{-4}.$$
\end{example}

\begin{proof}
To prove (\ref{ptohfk}) we match Theorem 5 with the equation (\ref{123}).

In the notation of Theorem 5 the non-symmetrized Alexander polynomial equals to
$$\Delta=\sum_{i=k}^{-k}(-1)^{k-i}t^{n_k-n_i}=\sum_{i=0}^{2k}(-1)^{i}t^{n_k-n_{k-i}},$$
$$P(t)={\Delta\over 1-t}=\sum_{i=0}^{k-1}\sum_{j=n_k-n_{k-2i}}^{n_k-n_{k-2i-1}-1}t^{j}+{t^{2n_k}\over 1-t}.$$
Note that for $i>0$ $$\delta_{k-2i}=\delta_{k-2i+1}-1=\delta_{k-2(i-1)}-2(n_{k-2i+2}-n_{k-2i+1}),$$
so
$$P_{g}(t,q)=\sum_{i=0}^{k-1}\sum_{j=n_k-n_{k-2i}}^{n_{k}-n_{k-2i-1}-1}q^{(j-n_{k}+n_{k-2i})-\delta_{k-2i}/2}t^{j}+{t^{2n_k}q^{n_k}\over 1-qt},$$
$$\overline{P}_{g}(t,q)=\sum_{i=0}^{k-1}(q^{-\delta_{k-2i}/2}t^{n_{k}-n_{k-2i}}-q^{-\delta_{k-2i-1}/2}t^{n_{k}-n_{k-2i-1}})+t^{2n_k}q^{n_k}.$$

Now $$\widetilde{\Delta_{g}}(t,u)=\sum_{i=0}^{k-1}(u^{\delta_{k-2i}}t^{n_{k}-n_{k-2i}}+u^{\delta_{k-2i-1}}t^{n_{k}-n_{k-2i-1}})+t^{2n_k}u^{-2n_k},$$
$$t^{n_k}\widetilde{\Delta_{g}}(t^{-1},u)=\sum_{i=0}^{k-1}(u^{\delta_{k-2i}}t^{n_{k-2i}}+q^{\delta_{k-2i-1}}t^{n_{k-2i-1}})+t^{2n_k}u^{-2n_k}=\sum_{i=-k}^{k}u^{\delta_{i}}t^{n_i}=HFL(t,u).$$

\end{proof}

\subsection{Comparing filtered complexes}

In this section we try to describe the relation between the knot filtration on the Heegaard-Floer complexes and the filtration on the space of functions defined by a curve.

To be more close to the algebraic setup, we reverse all signs for filtrations and for the homological (Maslov) grading as well (so we get cohomology groups). The Alexander grading is also changed to get the non-symmetrized Alexander polynomial. In another words, the Poincar\'e polynomial of the resulting cohomology
coincides with $\widetilde{\Delta_{g}}(t,u^{-1}).$ The operator $U$ will now increase the homological grading by 2.

Consider a $\mathbb{Z}_{\ge 0}$-indexed filtration $J_n$ by vector subspaces (with finite codimensions) on a infinite-dimensional complex vector space $J_0$. It induces a filtration by projective subspaces $\mathbb{P}J_{n}$ on $\mathbb{P}J_0=\mathbb{CP}^{\infty}$:
$$\mathbb{P}J_0\stackrel{j_1}{\hookleftarrow}\mathbb{P}J_1\stackrel{j_2}{\hookleftarrow}\mathbb{P}J_2\stackrel{j_3}{\hookleftarrow}\ldots, $$
so we have a sequence of corresponding Gysin maps in cohomology:
$$H^{*}(\mathbb{P}J_0)\stackrel{(j_1)_{*}}{\hookleftarrow}H^{*-2\cdot\mbox{\rm codim}J_1}\mathbb{P}J_1\stackrel{(j_2)_{*}}{\hookleftarrow}H^{*-2\cdot\mbox{\rm codim}J_2}\mathbb{P}J_2\stackrel{(j_3)_{*}}{\hookleftarrow}\ldots. $$

We get a $\mathbb{Z}_{\ge 0}$-indexed filtration $$F_{k}=(j_{k})_{*}(H^{*}(\mathbb{P}J_k))$$ in $H^{*}(\mathbb{CP}^{\infty})=\mathbb{Z}[U]$, which is compatible with the multiplication by $U$. If we also know (as for the filtration defined by the orders on the curve),
that $\dim J_{k}/J_{k+1}\le 1$, we conclude that $U$ increase the filtration level at least by 1.

The motivic Poincar\'e series in this setup can be written as 
$$P_{g}(t,q)=\sum_{k,n} t^{k}q^{n/2}\dim H^{n}(F_k/F_{k+1}).$$

The situation is similar to the Heegaard-Floer complexes, but  $U$ may increase the filtration level more that by 1. To avoid this problem, we should modify the complex.

\begin{example}
Consider the following filtered complex $T$: it has generators $U^{k}a_0$, $U^{k}a_1$ and $U^{k}a_2$.
The homological degree of $U^{l}a_j$ equals to $2l+j$ and its filtration level equals to $l+j$. 
The differential is defined as
$$d(a_1)=a_2+Ua_0.$$
One can check that  
$$\sum_{k,n}t^k u^n \dim H^{n}(T_k/T_{k+1})=1+u^2t^2+u^4t^3+u^6t^4+\ldots$$
(so this complex corresponds to minus-version of the Heegaard-Floer homology of the trefoil knot)
and $\mbox{\rm rk} H^{*}(T_k/UT_k)=1$ for all $k$. 
Remark that if $\widehat{T}^{k}=T_{k}/UT_{k-1}$, then 
$$\sum_{k,n}t^k u^n \dim H^{n}(\widehat{T}_k/\widehat{T}_{k+1})=1+ut+u^2t^2,$$
what is the Poincar\'e polynomial for the hat-version of the Heegaard-Floer homology of the trefoil.
\end{example}

\bigskip

Let us turn to the general case. Consider the complex 
\begin{equation}
\label{com}
\mathcal{C}_0=F_0[U_1]+(F_0[1])[U_1]
\end{equation}
with the filtration
$$\mathcal{C}_{n}=\bigoplus_{k+l=n}U_1^{l}F_{k}\oplus \bigoplus_{k+l=n-1}U_1^{l}F_{k}[1]$$
and the natural action of the operator $U_1$ of homological degree 2.
The differential is given by the equation
$$d(x)=U_1\cdot x+Ux.$$
One can check that this differential preserves the filtration $\mathcal{C}_n$ and commutes with $U_1$.

\begin{lemma}
$$H^{*}(\mathcal{C}_n/\mathcal{C}_{n+1})=F_n/F_{n+1},
\mbox{\rm rk}\quad H^{*}(\mathcal{C}_n/U_1(\mathcal{C}_{n}))=1.$$
\end{lemma}

\begin{proof}
We have
$$\mathcal{C}_{n}/\mathcal{C}_{n+1}=\bigoplus_{k+l=n}U_1^{l}(F_{k}/F_{k+1})\oplus \bigoplus_{k+l=n-1}U_1^{l}(F_{k}/F_{k+1})[1].$$
Since the $U_1$-increasing component of the differential
$$d_{1}(U_1^{l}x[1])=U_1^{l+1}x$$
gives the isomorphism
$$d_{1}:U_1^{l}(F_{k}/F_{k+1})\rightarrow U_1^{l+1}(F_k/F_{k+1}),$$
we have
$$H^{*}(\mathcal{C}_n/\mathcal{C}_{n+1})=F_n/F_{n+1}.$$

Also we have
$$\mathcal{C}_{n}/U_1(\mathcal{C}_n)=F_0\oplus F_0[1]\bigoplus_{k+l=n,\\ l>0}U_1^{l}(F_{k}/F_{k+1})\oplus \bigoplus_{k+l=n-1,\\ l>0}U_1^{l}(F_{k}/F_{k+1})[1],$$
and up to the isomorphisms $d_1$ we have the complex $F_0\oplus F_0[1]$ with the differential
$$d_{2}(x[1])=Ux,$$
so
$$rk\,\,\, H^{*}(\mathcal{C}_n/U_1(\mathcal{C}_{n}))=1.$$
\end{proof}

\bigskip

The properties of the complex $\mathcal{C}_0$ are similar to the ones of the complex $CFL^{-}(K)$. More precisely, the calculations of \cite{os} (lemma 3.1 and lemma 3.2) imply the following 

\begin{proposition}
Suppose that a cochain complex $\mathcal{C}$ has a filtration $\mathcal{C}_{k},\,\,\, k\ge 0$ and an injective operator  $U$ of homological degree 2 acting on it such that

1)$U(\mathcal{C}_{k})\subset C_{k+1}$ and $U^{-1}(\mathcal{C}_k)\subset \mathcal{C}_{k-1}$ (this means that $U$ increase the level of filtration exactly by 1) 
\medskip

2)$H^{*}(\mathcal{C}_{k}/U(\mathcal{C}_{k}))$ has rank 1 for all $k$.
\medskip

Then 

3) For all $k$ the rank of $H^{*}(\mathcal{C}_{k}/\mathcal{C}_{k+1})$ is at most 1.
\medskip

Let $\{0,\sigma_1,\sigma_2,\ldots \}$ is the set of $k$ such that this rank is 1.
Then

4) $H^{*}(\mathcal{C}_{\sigma_k}/\mathcal{C}_{\sigma_k+1})$ belongs to degree $2k$.
\medskip

Let $$Q(t,q)=\sum_{k=0}^{\infty}q^{k}t^{\sigma_k},\,\,\,\overline{Q}(t,q)=Q(t,q)(1-qt).$$
Let us make a following change  in $\overline{Q}$: $t^{\alpha}q^{\beta}$ is transformed to $t^{\alpha}u^{2\beta}$, and $-t^{\alpha}q^{\beta}$ is transformed to $t^{\alpha}u^{2\beta-1}.$ 

5) The result is equal to $$\sum_{k,n}t^{k}u^{n}\dim H^{n}(\mathcal{C}_{k}/(\mathcal{C}_{k+1}+U\mathcal{C}_{k-1})).$$
\end{proposition} 

\bigskip

The second condition is analogous to the equation (\ref{rk1}) for the Heegaard-Floer homology of the $L$-space knots.  

The last result can be reformulated as follows. Consider the complex $\widehat{\mathcal{C}}_{k}=\mathcal{C}_{k}/U\mathcal{C}_{k-1},$ then the last homology is the homology of the associated graded object $\widehat{\mathcal{C}}_{k}/\widehat{\mathcal{C}}_{k-1}$. The multiplication by $1-qt$ corresponds to the exact sequence
$$0\rightarrow \mathcal{C}_{k-1}/\mathcal{C}_{k}\stackrel{U}{\rightarrow}\mathcal{C}_{k}/\mathcal{C}_{k+1}\rightarrow \widehat{\mathcal{C}}_{k}/\widehat{\mathcal{C}}_{k+1}\rightarrow 0.$$

\bigskip

As a corollary we get that the series $Q(t,1)$ determines completely all discussed cohomology. Since for the filtered complexes $\mathcal{C}$ and $CFL^{-}$ we have $Q(t,1)=\Delta(t)/(1-t)$ for both, we have the equality of the cohomology of the associated graded objects and the more clear proof of the Theorem 6. As an another corollary, we get the equation 
\begin{equation}
\label{cs0}
H^{*}(CFL^{-}(S^3)/CFL^{-}_{s}(S^3,K))\cong H^{*}(\mathbb{P}(\mathcal{O}/J_s)),
\end{equation}
which looks more geometric than the Theorem 6.

\begin{remarks}

1. It would be interesting to construct the analogous $\mathbb{Z}^{n}$-filtered complex of $\mathbb{Z}[U_1,\ldots,U_n]$
for multi-component links which would carry the information about the Poincar\'e series of the corresponding multi-index filtration.


2. It would be also interesting to compare these results with the ones of \cite{n1}, \cite{n2} and \cite{n3} computing the Seiberg-Witten and Heegaard-Floer invariants of links of surface singularities. 
\end{remarks}

\subsection{Example: $A_{2n-1}$ singularities}

Since the algorithm of computation of the (reduced) motivic Poincar\'e series is quite complicated, it is useful to have a series of answers where the motivic Poincar\'e series and the Heegaard-Floer link homology can be computed. 

\begin{proposition}
Consider the singularity of type $A_{2n-1}$ given by the equation
$$y^2=x^{2n}.$$
From the topological viewpoint this corresponds to the 2-component link, whose components are unknotted, all intersections are positive and the linking number of the components equals to $n$. 
Then
$$P_{g}(t_1,t_2)=1+qt_1t_2+\ldots+q^{n-1}t_1^{n-1}t_2^{n-1}+{q^{n}(1-q)t_1^nt_2^n\over (1-t_1q)(1-t_2q)}.$$
\end{proposition}
\begin{proof}
For the proof we use the equation (\ref{motfil}). Parametrisations of the components are
$$(x(t_1),y(t_1))=(t_1,t_1^{n}),\,\,\,\,\mbox{\rm and}\,\,\,\,\, (x(t_2),y(t_2))=(t_2,-t_2^{n}),$$
so
$$x^ay^b|_{C_1}=t_1^{a+bn},\,\,\,\, x^ay^b|_{C_2}=(-1)^{b}t_2^{a+bn}.$$
If $a<n$, then every function with order $a$ on $C_1$ has a form $x^{a}+\ldots$, so
its order on $C_2$ is also equal to $a$.

For every $a,b \ge n$ consider the function $x^{a-n}(x^n+y)+x^{b-n}(x^n-y)$. Its restrictions on $C_1$ and $C_2$ are respectively equal to $2t_1^{a}$ and $2t_2^{b}$, therefore $$\dim J_{a,b}/J_{a+1,b}=\dim J_{a,b}/J_{a,b+1}=1.$$

The codimensions $h(v_1,v_2)$ are equal to $v_1+v_2-n$, if $v_1,v_2\ge n$, to $v_2$, if $v_1<n, v_2\ge n$, to
$v_1$, if $v_2<n, v_1\ge n$, and to $\max(v_1, v_2)$, if $0\le v_1,v_2<n$. We have 
$$L^{A_{2n-1}}_{g}(t_1,t_2,q)=\sum_{0\le \max(v_1,v_2);\min(v_1,v_2)< n}t_1^{v_1}t_2^{v_2}q^{\max(v_1,v_2)}+(1+q)\sum_{v_1,v_2=n}^{\infty}t_1^{v_1}t_2^{v_2}q^{v_1+v_2-n},$$
hence
$$L^{A_{2n-1}}_{g}(t_1-1)(t_2-1)=-1+(1-q)t_1t_2+\ldots+(q^{n-2}-q^{n-1})t_1^{n-1}t_2^{n-1}+q^{n-1}(1-q+q^2)t_1^{n}t_2^{n}$$
$$+{q^{n+1}t_1^{n+1}t_2^{n}(q-1)\over 1-qt_1}+{q^{n+1}t_1^{n}t_2^{n+1}(q-1)\over 1-qt_2}+{q^{n}t_1^{n+1}t_2^{n+1}(1+q)(1-q)^2\over (1-qt_1)(1-qt_2)},$$
and
$$P^{A_{2n-1}}_{g}={L^{A_{2n-1}}_{g}(t_1-1)(t_2-1)\over t_1t_2-1}=1+qt_1t_2+\ldots+q^{n-1}t_1^{n-1}t_2^{n-1}+{q^{n}(1-q)t_1^{n}t_2^{n}\over (1-qt_1)(1-qt_2)}.$$ 

\end{proof}

\begin{corollary}
\begin{equation}
\label{a2n+1}
\overline{P}_{g}^{A_{2n-1}}(t_1,t_2)=[1+(q+q^2)t_1t_2+\ldots+(q^{n-1}+q^{n})t_1^{n-1}t_2^{n-1}+q^{n}t_1^{n}t_2^{n}]
\end{equation}
$$-(t_1+t_2)[q+q^2t_1t_2+\ldots+q^{n}t_1^{n-1}t_2^{n-1}].$$
\end{corollary}

In \cite{os3} Ozsv\'ath and Szab\'o computed the Heegaard-Floer homology of the corresponding links.
In their notation the answer has the following form (everywhere we write the Poincar\'e polynomials of the corresponding complexes). Let
$$Y_{(d)}^{l}(t_1,t_2,u)=u^{d}(t_1^{l}+t_1^{l-1}t_2+\ldots+t_2^{l})+u^{d-1}(t_1^{l-1}+\ldots+t_2^{l-1}),$$
$$B_{(d)}(t_1,t_2,u)=u^{d}+(t_1+t_2)u^{d+1}+u^{d+2}t_1t_2.$$
Then
$$HFL_{A_{2n-1}}(t_1,t_2,u)=Y^{0}_{(0)}t_1^{n/2}t_2^{n/2}+Y^{1}_{(-1)}t_1^{n/2-1}t_2^{n/2-1}+\sum_{i=2}^{n}B_{(-2i)}t_1^{n/2-i}t_2^{n/2-i}.$$

Since $Y^{0}_{(0)}=1$ and $Y^{1}_{(-1)}=u^{-1}(t_1+t_2)+u^{-2}$ one can simplify this as
$$HFL_{A_{2n-1}}(t_1,t_2,u)=t_1^{n/2}t_2^{n/2}+(u^{-1}(t_1+t_2)+u^{-2})t_1^{n/2-1}t_2^{n/2-1}$$
$$+\sum_{i=2}^{n}(u^{-2i}+(t_1+t_2)u^{-2i+1}+u^{-2i+2}t_1t_2)t_1^{n/2-i}t_2^{n/2-i},$$
so
$$t_1^{n/2}t_2^{n/2}HFL_{A_{2n-1}}(t_1^{-1},t_2^{-1},u)=1+(u^{-1}(t_1+t_2)+u^{-2}t_1t_2)$$
$$+\sum_{i=2}^{n}(u^{-2i}t_1^{i}t_2^{i}+(t_1+t_2)u^{-2i+1}t_1^{i-1}t_2^{i-1}+u^{-2i+2}t_1^{i-1}t_2^{i-1})=$$
$$[1+2u^{-2}t_1t_2+\ldots+2u^{-2n+2}t_1^{n-1}t_2^{n-1}+u^{-2n}t_1^{n}t_2^{n}]$$
$$-(t_1+t_2)[u^{-1}+u^{-3}t_1t_2+\ldots+u^{-2n+1}t_1^{n-1}t_2^{n-1}].$$

\bigskip

The last expression is similar to (\ref{a2n+1}) in analogy with the Theorem 6. 


\section{Appendix}

{\bf Proof of Lemma \ref{L4}. }

We have
$$\sum u^{\hat{n}_i}\phi_i(I,K,\hat{n})=\sum_j\sum_{\hat{n}=j+f_i(K,I)}^{\infty} u^{\hat{n}_i}
(-1)^{j}{1-\chi(E_i^{\circ})-f_i(I,K)\choose j}q^{j}=$$
$${u^{f_i(K,I)}\over 1-u}\sum_{j}(-1)^{j}{1-\chi(E_i^{\circ})-f_i(I,K)\choose j}(uq)^{j}={u^{f_i(K,I)}\over 1-u}(1-uq)^{1-\chi(E_i^{\circ})-f_i(I,K)},$$
and
$$\sum u^{\hat{n}}G(K,I,\hat{n})=q^{|I|}(1-q)^{|I|+|K|}\prod_{i}{u_i^{f_i(K,I)}\over 1-u_i}(1-u_iq)^{1-\chi(E_i^{\circ})-f_i(I,K)}.$$

$\square$

{\bf Proof of Lemma \ref{L5}}

$$A_K(u)=\sum_{I}q^{|I|}(1-q)^{|I|}\sum_{K_1}(-1)^{|K|-|K_1|}(1-q)^{|K_1|}\sum_n u^{n}\prod_i\phi_i(I,K_1,n).$$
We have
$$\sum_n u^{n}\prod_i\phi_i(I,K_1,n)=\prod_{i}{u_i^{f_i(K,I)}(1-u_iq)^{1-\chi(E_i^{\circ})-f_i(I,K)}\over 1-u_i}.$$
Now
$$\sum_{K_{1i}\subset (K\cap E_i)}(-1)^{|K\cap E_i|-|K_{i1}|}(1-q)^{|K_{1i}|}{1\over 1-u_i}u_i^{f_i(K_1,I)}(1-u_i q)^{1-\chi(E_i^{\circ})-f_i(I,K_1)}=$$
$${1\over 1-u_i}u_i^{f_i(K,I)}(1-u_i q)^{1-\chi(E_i^{\circ})-f_i(K,I)}\times$$
$$\sum_{K_{1i}}(-1)^{|K\cap E_i|-|K_{1i}|}(1-q)^{|K_{1i}|}u_i^{|K_{1i}|-|K\cap E_i|}(1-u_i q)^{|K\cap E_i|-|K_{1i}|}=$$
$${1\over 1-u_i}u_i^{f_i(K,I)}(1-u_i q)^{1-\chi(E_i^{\circ})-f_i(K,I)}(1-q-{1-u_i q\over u_i})^{|K\cap E_i|}=$$
$${1\over 1-u_i}(-1)^{|K\cap E_i|}u_i^{f_i(K,I)-|K\cap E_i|}(1-u_i q)^{1-\chi(E_i^{\circ})-f_i(K,I)}(1-u_i)^{|K\cap E_i|}.$$
Remark that $f_i(K,I)-|K\cap E_i|=f_i(I)$ and $$\chi(E_i^{\circ})+f_i(K,I)=\chi(E_i^{\bullet})-|K_0\cap E_i|+|K\cap E_i|+f_i(I),$$ hence the last expression can be rewritten in a form
$$(-1)^{|K\cap E_i|}u_i^{f_i(I)}(1-u_i q)^{1-\chi(E_i^{\bullet})+|\overline{K}\cap E_i|-f_i(I)}(1-u_i)^{|K\cap E_i|-1}.$$
Also
$$\sum_{I}q^{|I|}(1-q)^{|I|}\prod_i u_i^{f_i(I)}(1-u_iq)^{-f_i(I)}=\prod_{\sigma}(1+q(1-q)u_{i(\sigma)}u_{j(\sigma)}(1-u_{i(\sigma)}q)^{-1}(1-u_{j(\sigma)}q)^{-1})=$$
$$\prod_{i}(1-u_iq)^{\chi(E_i^{\bullet})-2}\prod_{\sigma}(1-qu_{i(\sigma)}-qu_{j(\sigma)}+qu_{i(\sigma)}u_{j(\sigma)}).$$
Therefore
$$A_{K}(u)=(-1)^{|K|}\prod_{i}(1-u_iq)^{1-\chi(E_i^{\bullet})+|\overline{K}\cap E_i|}(1-u_i)^{|K\cap E_i|-1}\times$$
$$\times\prod_{i}(1-u_iq)^{\chi(E_i^{\bullet})-2}\prod_{\sigma}(1-qu_{i(\sigma)}-qu_{j(\sigma)}+qu_{i(\sigma)}u_{j(\sigma)})=$$
$$(-1)^{|K|}\prod_{i}(1-u_i q)^{|\overline{K}\cap E_i|-1}(1-u_i)^{|K\cap E_i|-1}
\prod_{\sigma}(1-qu_{i(\sigma)}-qu_{j(\sigma)}+qu_{i(\sigma)}u_{j(\sigma)}).$$

$\square$

{\bf Proof of Lemma \ref{L7}}

We have to prove that $\widetilde{H}_{P}=0$ at $u_{\beta}=1$ for $\beta\in E(P)$.
Suppose that $E_{\beta}$ is intersected by $E_{\alpha_1},\ldots,E_{\alpha_k}$. For every set $E$ of divisors not containing $E_{\beta}$ let us compare the summands corresponding to $E$ and to $E\cup E_{\beta}$. 

For $E$ at $u_{\beta}=1$ we have 
$$\prod_{i\neq \beta} u_i^{-\sum a_{ij}\mu_j}(-1)^{|K_0\cap E|}q^{\Delta(E)}\prod_{i\in E}(q-u_i)^{k_i-1}
(1-q)^{k_{\beta}-1}\prod_{i\notin (P\cup E)}(1-qu_i)^{k_i-1}$$
$$\times \prod_{\sigma\notin E_{\beta}}(1-q^{1-\mu_{i(\sigma)}(E)}u_{i(\sigma)}-q^{1-\mu_{j(\sigma)}(E)}u_{j(\sigma)}+q^{1-\mu_{i(\sigma)}(E)-\mu_{j(\sigma)}(E)}u_{i(\sigma)}u_{j(\sigma)})\cdot (1-q)^{k}.$$

For $E\cup E_1$ at $u_{\beta}=1$ we have
$$\prod_{j=1}^{k}u_{\alpha_j}\prod_{i\neq \beta} u_i^{-\sum a_{ij}\mu_j}(-1)^{k_{\beta}+|K_0\cap E|}q^{\Delta(E\cup E_1)}(q-1)^{k_{\beta}-1}\prod_{i\in E}(q-u_i)^{k_i-1}
\prod_{i\notin (E\cup P)}(1-qu_i)^{k_i-1}$$
$$\times \prod_{\sigma\notin E_{\beta}}(1-q^{1-\mu_{i(\sigma)}(E)}u_{i(\sigma)}-q^{1-\mu_{j(\sigma)}(E)}u_{j(\sigma)}+q^{1-\mu_{i(\sigma)}(E)-\mu_{j(\sigma)}(E)}u_{i(\sigma)}u_{j(\sigma)})\cdot \prod_{j=1}^{k}(1-q)q^{-\mu_{\alpha_{j}}(E)}u_{\alpha_j}.$$

It rests to note that $\Delta(E\cup E_{\beta})-\Delta(E)=\sum_{j=1}^{k}\mu_{\alpha_{j}}(E)$.

$\square$

{\bf Proof of Lemma \ref{L8}.}

$$\sum_{n}u^n\sum_{E\subset E(P)}q^{-\sum_{i\in E}n_i-\Delta(E)-\sum_{i\in E}a_{ii}-|E|}q^{|K_0\cap E|}\times c_{P\cup E}(n_i+\sum a_{ij}\mu_j(E))=$$ 
$$\sum_{E\subset E(P)}\prod u_i^{-\sum a_{ij}\mu_j(E)}\cdot q^{\sum a_{ij}\mu_i(E)\mu_j(E)}\cdot q^{-\Delta(E)-\sum_{i\in I}a_{ii}+|K_0\cap E|-|E|}$$
$$\times\sum_{n_1} \prod_{i}(u_iq^{-\mu_i(E)})^{n_{1i}}\cdot c_{P\cup E}(n_1)=$$
$$\sum_{E\subset E(P)}\prod u_i^{-\sum a_{ij}\mu_j(E)}\cdot A_{P\cup E}(u_iq^{-\mu_i(E)})q^{\Delta(E)+|K_0\cap E|-|E|}=$$
$$(-1)^{|P|}\sum_{E\subset E(P)}\prod u_i^{-\sum a_{ij}\mu_j(E)}\cdot(-1)^{|K_0\cap E|}q^{\Delta(E)+|K_0\cap E|-|E|}\prod_{i\in E}[(1-u_i)^{-1}(1-u_iq^{-1})^{k_i-1}]$$
$$\times\prod_{i\in P}[(1-qu_i)^{k_i-p_i-1}(1-u_i)^{p_i-1}]\prod_{i\notin (P\cup E)}[(1-qu_i)^{k_i-1}(1-u_i)^{-1}]$$
$$\times\prod_{\sigma}(1-q^{1-\mu_{i(\sigma)}(E)}u_{i(\sigma)}-q^{1-\mu_{j(\sigma)}(E)}u_{j(\sigma)}+q^{1-\mu_{i(\sigma)}(E)-\mu_{j(\sigma)}(E)}u_{i(\sigma)}u_{j(\sigma)})=$$
$$(-1)^{|P|}\prod_{i\in P}[(1-qu_i)^{k_i-p_i-1}(1-u_i)^{p_i-1}]\cdot{1\over \prod_{i\in E(P)}(1-u_i)}$$
$$\times\sum_{E\subset E(P)}(-1)^{|K_0\cap E|}\cdot \prod u_i^{-\sum a_{ij}\mu_j(E)}\cdot q^{\Delta(E)}\prod_{i\in E}(q-u_i)^{k_i-1}
\prod_{i\notin E}(1-qu_i)^{k_i-1}$$ $$\times \prod_{\sigma}(1-q^{1-\mu_{i(\sigma)}(E)}u_{i(\sigma)}-q^{1-\mu_{j(\sigma)}(E)}u_{j(\sigma)}+q^{1-\mu_{i(\sigma)}(E)-\mu_{j(\sigma)}(E)}u_{i(\sigma)}u_{j(\sigma)}).$$

$\square$

{\bf Proof of Theorem \ref{T4}.}

Let $k_i=|K_0\cap E_i|.$
From Lemma \ref{L6} we get
$$\overline{P}_{g}({1\over qt_1},\ldots,{1\over qt_r})=(t_1\cdot\ldots\cdot t_r)^{-1}\sum_{\underline{n}} \underline{t}^{-M\underline{n}}q^{-\sum m_{ij}k_in_j}q^{F(n)-\sum n_i}\sum_{K}t_{\overline{K}}c_{K}(n)=$$
$$\underline{t}^{-\underline{1}-M\chi(\underline{E}^{\circ})}\sum_{\underline{n}} \underline{t}^{M(\chi(\underline{E}^{\circ})-\underline{n})}q^{-\sum m_{ij}k_in_j}q^{F(n)-\sum n_i}$$
\begin{equation}
\label{tmp}
\times \sum_{K} q^{1-|K|+n}\cdot t_{\overline{K}}\cdot c_{\overline{K}}(-\chi(E_i^{\circ})-n_i).
\end{equation}
Let $$\xi_i=-\chi(E_i^{\circ}),\,\,\,\,\,\, \underline{n}_1=\underline{\xi}-\underline{n}.$$ Then
$$F(n)-\sum n_i={1\over 2}[\sum m_{ij}n_in_j+\sum m_{ij}n_i\chi(E_{j}^{\bullet})-\sum n_i],$$
so
$$2[F(n_1)-\sum n_{1i}-F(n)+\sum n_{i}]=$$
$$\sum m_{ij}(\xi_i-n_i)(\xi_j-n_j)+\sum m_{ij}(\xi_i-n_i)\chi(E_j^{\bullet})-\sum (\xi_i-n_i)$$
$$-\sum m_{ij}n_in_j-\sum m_{ij}n_i\chi(E_{j}^{\bullet})+\sum n_i=$$
$$-2\sum m_{ij}(\xi_i+\chi(E_{i}^{\bullet})) n_j+2\sum n_j+2(F(\xi)-\sum\xi_i)=$$
$$-2\sum m_{ij}k_in_j+2\sum n_j+2(F(\xi)-\sum\xi_i).$$
Thus (\ref{tmp}) is equal to
$$t^{-1-M\xi}q^{-F(\xi)+\sum \xi_i}q^{1-|K_0|}\sum t^{Mn_1}q^{F(n_1)-\sum n_{1i}}\sum_{K}t_{\overline{K}}q^{|\overline{K}|}c_{\overline{K}}(n_1).$$

\medskip

It rests to compute the powers of $t_{\alpha}$ and of $q$.

Remark that $\sum \xi_i=|K_0|-2,$
so
$\sum \xi_i+1-|K_0|=-1.$

Also $$2F(\xi)=\sum m_{ij} k_ik_j-2\sum m_{ij}k_i\chi(E_{j}^{\bullet})+\sum m_{ij} \chi(E_{i}^{\bullet})\chi(E_j^{\bullet})+$$
$$\sum m_{ij} k_i\chi(E_j^{\bullet})-\sum m_{ij}\chi(E_i^{\bullet})\chi(E_j^{\bullet})+\sum \xi_i=$$
$$\sum m_{ij} k_i k_j-\sum m_{ij} k_i\chi(E_j^{\bullet})+|K_0|-2.$$
The formula of A'Campo (\cite{acampo}) says that
$$1-\mu=\sum m\chi(S_m)=\sum\chi(E_{i}^{\circ})m_{ij}k_j=\sum m_{ij} (\chi(E_{i}^{\bullet})-k_i)k_j,$$
so
$$2F(\xi)=\mu-1+|K_0|-2=2\delta-2.$$
Thus $-F(\xi)-1=-\delta.$  

Also for every $\alpha$ one has
$$1-\mu_{\alpha}=\sum_{j\neq i(\alpha)}m_{i(\alpha)j}\chi(E_j^{\bullet})+m_{i(\alpha),i(\alpha)}(\chi(E_{i(\alpha)}^{\bullet})-1),$$
and for $\beta\neq \alpha$
$$C_{\alpha}\circ C_{\beta}=m_{i(\alpha),i(\beta)},$$
so
$$\sum_{\beta\neq\alpha}C_{\alpha}\circ C_{\beta}=\sum_{j\neq i(\alpha)}m_{i(\alpha),j}k_j+m_{i(\alpha),i(\alpha)}(k_{i(\alpha)}-1)$$
and
$$1-\mu_{\alpha}-C_{\alpha}\circ C_{\beta}=\sum_{j}m_{i(\alpha),j}\chi(E_{j}^{\circ}).$$

$\square$


State University of New York at Stony Brook,\newline
Department of Mathematics\newline
E. mail: \tt{gorsky@mccme.ru, egorsky@math.sunysb.edu}.

\end{document}